\documentclass[11pt,letterpaper]{amsart}

\usepackage[margin=1.4in]{geometry}
\usepackage{url}
\usepackage[colorlinks=true, citecolor=blue]{hyperref}
\usepackage{bookmark}
    \usepackage{amssymb}
\usepackage{amsmath}
\usepackage{amscd,latexsym}
\usepackage{bookmark}
     \usepackage{amssymb,amsfonts,bm}
\usepackage[all,arc]{xy}
\usepackage{enumerate}
\usepackage{mathrsfs}
\usepackage{amscd}
\usepackage{tikz}
\usepackage{tikz-cd}
\usepackage{latexsym}
\usepackage{mathdots}
\usepackage{amsrefs}
\usepackage{graphicx}
\usepackage{mathtools}
\usepackage{leftidx}
\usepackage{tensor}
\usepackage{bbm}
\usepackage{dsfont}
\usepackage{xypic,amsthm, color}
\usepackage[new]{old-arrows}
\usepackage{enumitem,xcolor,tabmacD}
\usepackage{mdwlist}
\usepackage{multicol}
\usepackage{xparse}
\newcommand{\Tableau}[2][sY]{{\text{\tableau[#1]{#2}}}}
\usetikzlibrary{arrows.meta,knots}

\makeatletter
\@namedef{subjclassname@2020}{%
  \textup{2020} Mathematics Subject Classification}
\makeatother

\theoremstyle{plane}
\newtheorem{theorem}{Theorem}[section]
\newtheorem{lemma}[theorem]{Lemma}
 \newtheorem{corollary}[theorem]{Corollary}
      \newtheorem{proposition}[theorem]{Proposition}

        \newtheorem{theorem*}{Theorem}

    \newcounter{relctr} 
\everydisplay\expandafter{\the\everydisplay\setcounter{relctr}{0}} 

\AtBeginDocument{} 

\NewDocumentCommand\DownArrow{O{2.0ex} O{black}}{%
   \mathrel{\tikz[baseline] \draw [<-, line width=0.5pt, #2] (0,0) -- ++(0,#1);}
}

\theoremstyle{definition}
 \newtheorem*{definition*}{Definition}
\newtheorem{definition}[theorem]{Definition}

 \newtheorem*{data availability}{Data Availability}

\theoremstyle{remark}

\newtheorem*{acknowledgments}{Acknowledgments}

\numberwithin{equation}{section}

\begin{document}
  \title[Two-pointed Prym-Brill-Noether Loci and coupled Prym-Petri theorem]{Two-pointed Prym-Brill-Noether Loci and coupled Prym-Petri theorem}
  \author{Minyoung Jeon}
\address{Department of mathematics, University of Georgia, Athens GA 30602, USA}
\email{minyoung.jeon@uga.edu}

\subjclass[2020]{Primary 14H40, 14H51, 14M15; Secondary 14H10, 14C25, 19E20} 
\keywords{Two-pointed Prym-Brill-Noether loci, coupled Prym-Petri map, vexillary degeneracy loci, K-theory classes}

   \begin{abstract}
 We establish two-pointed Prym-Brill-Noether loci with special vanishing at two points, and determine their K-theory classes when the dimensions are as expected. The classes are derived by the applications of a formula for the K-theory of certain vexillary degeneracy loci in type D. In particular, we show a two-pointed version of Prym-Petri theorem on the expected dimension in the general case, with a coupled Prym-Petri map. Our approach is inspired by the work on pointed cases by Tarasca, and we generalize unpointed cases by De Concini-Pragacz and Welters.

         \end{abstract}

   \date{\today}

   \maketitle

\section{Introduction}
Prym varieties belong to a class of principally polarized abelian varieties associated with \'etale double covers of algebraic curves. The Prym-Brill-Noether loci, which are the Brill-Noether loci for Prym varieties, were introduced by Welters \cite{Wel85}. Later, these loci were generalized to the Prym-Brill-Noether loci with special vanishing at one point by the author \cite{Jeon} and Tarasca \cite{Tar}, so-called, pointed Prym-Brill-Noether loci. The scheme structure can be understood as degeneracy loci of Lie type D. To be specific, De Concini-Pragacz \cite{DCP} regarded the Prym-Brill-Noether loci as the degeneracy loci, characterized by a rank condition on the intersection of two maximal isotropic subbundles in a vector bundle with respect to a non-degenerate quadratic form. The pointed cases \cite{Jeon, Tar} were presented as the degeneracy loci by rank conditions on the intersection of a maximal isotropic subbundle of the vector bundle and an isotropic flag associated to a fixed point.

The goal of this article is to study the Prym-Brill-Noether loci with special vanishing at two points, providing some analogous results from the pointed case in \cite{Jeon,Tar}. We not only identify the two-pointed Prym-Brill-Noether loci as being associated with a certain vexillary degeneracy loci of Lie type D in expected dimension, but also give a formula for the class of the two-pointed Prym-Brill-Noether loci in connective K-theory. Furthermore we establish that the expected dimensions of the two-pointed Prym-Brill-Noether loci hold generically.

In concrete terms, the {\it two-pointed Prym-Brill-Noether locus} is defined as follows. Let $C$ be a smooth algebraic curve of genus $g$ over an algebraically closed field $\mathbb{K}$ of characteristic different from two.
Let $\pi:\widetilde{C}\rightarrow C$ be an \'etale double cover of $C$, which is determined by a nontrivial $2$-torsion point $\epsilon$ in $\mathrm{Jac}(C)$. Let $
\mathbf{a}'=(0\leq a'_0< a'_1< \cdots<a'_r\leq 2g-2)$ and $\mathbf{b}'=(0\leq b'_0< b'_1< \cdots<b'_r\leq 2g-2)
$
be strictly increasing sequences such that 
\begin{equation}\label{con}
\min\{a'_{i+1}-a'_i,b'_{i+1}-b'_i\}=1
\end{equation} for $i=0,\ldots, r-2$. For points $P$ and $Q$ in $\widetilde{C}$, the locus $V_{\mathbf{a}',\mathbf{b}'}^r(C,\epsilon,P,Q)$ in a Prym variety $\mathscr{P}^+$ for odd $r$ (or $\mathscr{P}^-$ for even $r$) is defined to be 
\begin{align*}
 V_{\mathbf{a}',\mathbf{b}'}^r(C,\epsilon,P,Q):=\left\{L\in \right.& \mathscr{P}^\pm\;|\;h^0(\widetilde{C},L)\equiv r+1\;\text{(mod 2)},\\
&\left. h^0(\widetilde{C}, L(-b'_{j}Q'-a'_{i}P')))\geq r+1-j-i\;\text{for all $i,j$}\right\}.\nonumber
\end{align*}
Indeed, set-theoretically, we can recover Prym-Brill-Noether loci from Welters \cite{Wel85} by taking $\mathbf{a}'=(0,\ldots,r)$ and $\mathbf{b}'=(0,\ldots, r)$, and one may forget either a pair $(P,\mathbf{a})$ or $(Q,\mathbf{b}')$ with the condition \eqref{con} to obtain the pointed Prym-Brill-Noether loci in \cite[\S 4.1]{Jeon}, \cite[\S 2]{Tar}. From an analysis of degeneracy loci of type D in \S\ref{sec2}, the conditions specified by $\mathbf{a}'$ and $\mathbf{b}'$ for $V_{\mathbf{a}',\mathbf{b}'}^r(C,\epsilon,P,Q)$ are equivalent to the conditions associated to modified sequences $\mathbf{a}$ and $\mathbf{b}$ for $V_{\mathbf{a},\mathbf{b}}^r(C,\epsilon,P,Q)$ defined by 
\[
 h^0(\widetilde{C},L(-b_iQ-a_iP))\geq r+1-i\quad\text{ for $i=0,\ldots,r$}
 \]
in $\mathscr{P}^\pm$. This enables us to impose the scheme structure of a vexillary Degeneracy loci of type D on the two-pointed Prym-Brill-Noether loci, generalizing the case $\mathbf{a}'=(0,\ldots,r)$ and $\mathbf{b}'=(0,\ldots, r)$ in \cite{DCP}. When the two points $P$ and $Q$ come together, the locus $V_{\mathbf{a},\mathbf{b}}^r(C,\epsilon,P,Q)$ specializes to a pointed Prym-Brill-Noether locus associated to a strictly increasing sequence $a_i+b_i$ for $i=0,\ldots,r$, and it is endowed with the scheme structure in \cite{Jeon,Tar}.

 The connective K-theory introduced by Cai \cite{Cai} for schemes establishes a connection between the Chow groups and Quillen's K-theory groups. Dai and Levine \cite{DL} explored this notion within the realm of motivic homotopy theory. In our work, we employ a simpler variant of the connective K-theory for the scheme: the connective K-cohomology $CK^*(X)$ for an irreducible variety $X$. See \cite{And19,HIMN,HIMN20} especially for more details in the context in degeneracy loci.

 When $V_{\mathbf{a},\mathbf{b}}^r(C,\epsilon,P,Q)$ has the expected dimension $g-1-|\mathbf{a}+\mathbf{b}|$, the K-theory class formulas for $V_{\mathbf{a},\mathbf{b}}^r(C,\epsilon,P,Q)$ in connective K-theory are stated as follows. Let $\lambda$ be a partition such that 
 \[
 \lambda_{i}=a_{r+1-i}+b_{r+1-i},
 \]
and $\ell_\circ:=\ell(\lambda)$ be the length of $\lambda$, the number of non-zero parts of $\lambda$.

\begin{theorem}[= Theorem \ref{thm:ck}]\label{main}
The dimension of $V_{\mathbf{a},\mathbf{b}}^r(C,\epsilon,P,Q)$ is at least $g-1-|\mathbf{a}+\mathbf{b}|$. If $V_{\mathbf{a},\mathbf{b}}^r(C,\epsilon,P,Q)$ has dimension of $g-1-|\mathbf{a}+\mathbf{b}|$, then it is Cohen-Macaulay, and 
\begin{align*}
\left[V_{\mathbf{a},\mathbf{b}}^r(C,\epsilon,P,Q)\right]&=Pf_\lambda(d(1),\ldots,d(\ell_\circ);\beta)
\end{align*}
in $CK^*(\mathscr{P}^\pm)[1/2]$. 
\end{theorem}
The above formula is expressed by a Pfaffian of a skew-symmetric matrix associated to K-theoretic Chern classes, as derived from the K-theory class formulas for vexillary degeneracy loci of type D in \cite[Theorem 4]{And19}. Unspecified notations in the theorem will be defined later in \S\ref{sec3}. As a corollary, we have the following about the class of $V_{\mathbf{a},\mathbf{b}}^r(C,\epsilon,P,Q)$ in the numerical equivalence ring $N^*(\mathscr{P}^\pm,\mathbb{K})$ or the singular cohomology ring $H^*(\mathscr{P}^\pm,\mathbb{C})$. Let $\xi$ be the class of theta divisor on the Prym variety $\mathscr{P}^\pm$ in $N^*(\mathscr{P}^\pm,\mathbb{K})$ or in $H^*(\mathscr{P}^\pm,\mathbb{C})$.

\begin{corollary}\label{cor33}
If $\mathrm{dim}(V_{\mathbf{a},\mathbf{b}}^r(C,\epsilon,P,Q))= g-1-|\mathbf{a}+\mathbf{b}|$, then $V_{\mathbf{a},\mathbf{b}}^r(C,\epsilon,P,Q)$ has the class
\begin{equation}\label{class}
\left[V_{\mathbf{a},\mathbf{b}}^r(C,\epsilon,P,Q)\right]=\dfrac{1}{2^{\ell_\circ}}\prod_{i=0}^{r}\dfrac{1}{{(a_i+b_i)!}}\prod_{j<i}\dfrac{a_i+b_i-a_j-b_j}{a_i+b_i+a_j+b_j}\cdot (2\xi)^{|\mathbf{a}+\mathbf{b}|}
\end{equation}
in $N^*(\mathscr{P}^\pm,\mathbb{K})$ or $H^*(\mathscr{P}^\pm,\mathbb{C})$.
\end{corollary}

In particular, the right hand side of \eqref{class} is supported on $V_{\mathbf{a},\mathbf{b}}^r(C,\epsilon,P,Q)$ even if the dimension of $V_{\mathbf{a},\mathbf{b}}^r(C,\epsilon,P,Q)$ is greater than $g-1-|\mathbf{a}+\mathbf{b}|$. Moreover, with $\xi$ being ample, the right hand side of \eqref{class} becomes non-zero whenever $g-1-|\mathbf{a}+\mathbf{b}|\geq 0$. This implies that if $g-1\geq|\mathbf{a}+\mathbf{b}|$, the two-pointed Prym-Brill-Noether locus $V_{\mathbf{a},\mathbf{b}}^r(C,\epsilon,P,Q)$ is not empty and has dimension at least $g-1-|\mathbf{a}+\mathbf{b}|$, for any such a pair $(C,\epsilon,P,Q)$. Corollary \ref{cor33} recovers the results in \cite[Theorem 9]{DCP} with $\mathbf{a}'=(0,1,\ldots,r)$ and $\mathbf{b}'=(0,1,\ldots,r)$, and a two-pointed version analogous to \cite[Theorem 1]{Tar}.
The subsequent theorem confirms that the expected dimension of $V_{\mathbf{a},\mathbf{b}}^r(C,\epsilon,P,Q)$ is achieved generically.

\begin{theorem}\label{main:2}
Let $C$ be a general curve of genus $g$, with its \'etale double cover $\widetilde{C}$ associated to $\epsilon$, an arbitrary non-trivial 2-torsion point in $\mathrm{Jac}(C)$. Let $P$ and $Q$ be general points in $\widetilde{C}$. Then $V_{\mathbf{a},\mathbf{b}}^r(C,\epsilon,P,Q)$ is either empty or has dimension $g-1-|\mathbf{a}|-|\mathbf{b}|$ at $L\in V_{\mathbf{a},\mathbf{b}}^r(C,\epsilon,P,Q)$ such that  $h^0(\widetilde{C},L(-a_iP-b_iQ))= r+1-i\;\text{for all $i=0,\ldots,r$}$.
\end{theorem}

We obtain the main theorem presented above using arguments similar to \cite{EH83,Tar,Wel85}, which are used in the proof of Gieseker-Petri type theorem. Theorem \ref{main:2} encompasses the result from \cite{Wel85} concerning the Prym-Brill-Noether loci. 
Consequently, we derive conditions for the non-emptiness of the pointed Prym-Brill-Noether loci $V_{\mathbf{a},\mathbf{b}}^r(C,\epsilon,P,Q)$, as follows.

\begin{corollary}\label{cor}
Let $C$ be a general curve with its genus equals to $g$. Let $\epsilon\in\mathrm{Jac}(C)$ be an arbitrary $2$-torsion point, and $P,Q\in\widetilde{C}$ general points. Then $V_{\mathbf{a},\mathbf{b}}^r(C,\epsilon,P,Q)\neq\emptyset$ if and only if $g-1\geq |\mathbf{a}|+|\mathbf{b}|$. 

In particular, if $V_{\mathbf{a},\mathbf{b}}^r(C,\epsilon,P,Q)$ is nonempty, the dimension of $V_{\mathbf{a},\mathbf{b}}^r(C,\epsilon,P,Q)$ is $g-1-|\mathbf{a}|-|\mathbf{b}|$ and its class is equal to the class \eqref{class} in $N^*(\mathscr{P}^\pm,\mathbb{K})$ and $H^*(\mathscr{P}^\pm,\mathbb{C})$, where the characteristic of $\mathbb{K}$ is not equal to $2$.
\end{corollary}

Lastly, our study focuses on the two-pointed Prym-Brill-Noether loci under the condition \eqref{con}, which allows us to establish Theorem \ref{main:2}. In light of constraint \eqref{con}, we take a partial step toward expanding the results on pointed Prym-Brill-Noether loci by the author \cite{Jeon} and Tarasca \cite{Tar} to the two-pointed cases. So, it would be interesting to investigate further on the two-pointed Prym-Brill-Noether loci without \eqref{con}, as a complete generalization of pointed Prym-Brill-Noether loci. Even more, the study of any motivic version of Prym-Brill-Noether loci with special vanishing at up to two fixed points could be of considerable interest. 

The organization of this manuscript is outlined in the following manner. In \S \ref{sec2} we discuss the degeneracy loci of Lie type D associated to two strictly increasing sequences $\mathbf{a}'$ and $\mathbf{b}'$ satisfying \eqref{con} in general, and show the connection to the degeneracy loci associated to the new sequences $\mathbf{a}$ and $\mathbf{b}$. In \S\ref{sec3} we define the two-pointed Prym-Brill-Noether loci and establish formulas for the K-theory class of the loci in connective K-theory. In \S\ref{sec4} we determine a coupled Prym-Petri map, and \S\ref{sec5} contains the proof of the coupled version of the Prym-Petri Theorem, showing Theorem \ref{main:2}.

\begin{acknowledgments}
The author expresses gratitude to David Anderson for suggesting the consideration of an analogy to Brill-Noether loci, specifically concerning the vanishing properties at two points for Prym varieties. Additionally, the author acknowledges the partial support received through the American Mathematical Society (AMS)-Simons travel grant.
\end{acknowledgments}

\section{The Degeneracy loci}\label{sec2}
We consider sequences of integers 
$
\mathbf{p}'=(0\leq p'_0<p'_1<\cdots<p'_s)\;\text{and}\;
\mathbf{q}'=(0\leq q'_0<q'_1<\cdots<q'_s),
$
satisfying 
\[
\min\{p'_{i+1}-p'_i,q'_{i+1}-q'_i\}=1
\]
 for $i=0,\ldots, s-2$. 

Let $V$ be a vector bundle of rank $2n$ over a variety $X$, equipped with a non-degenerate quadratic form. We consider flags of isotropic subbundles  
\begin{align*}
E_{p'_s}\subset E_{p'_{s-1}}\subset\cdots\subset E_{p'_0}\subset V\quad\text{and}\quad
F_{q'_s}\subset F_{q'_{s-1}}\subset\cdots\subset F_{q'_0}\subset V
\end{align*} 
on $X$, where the rank of $E_{p'}$ is $n-p'$ and $F_{q'}$ is $n-q'$. 
The locus $V_{\bold{p}',\bold{q}'}$ associated to the sequences $\mathbf{p}'$ and $\mathbf{q}'$ is given by 
\begin{equation}\label{eqn:rk}
\mathrm{dim}(E_{p'_i}\cap F_{q'_j})\geq s+1-j-i\quad\text{for all}\; i,j.
\end{equation}
This degeneracy locus (and any locus defined by such conditions in this manuscript) must be read by the closure of the locus where equality holds. We further impose an additional condition either $\mathrm{dim}(E_0\cap F_0)\equiv 0$ (mod $2$) or $\mathrm{dim}(E_0\cap F_0)\equiv 1$ (mod $2$) on $V_{\mathbf{p}',\mathbf{q}'}$. It is worthwhile noting that the conditions for $i+j\leq s$ are enough to define the locus, since the remaining ones become trivial. 

Let us analyze the scheme structure of $V_{\mathbf{p}',\mathbf{q}'}$ with some combinatorics and notions about Schubert varieties of Lie type D in \cite[\S 2]{AIJK}. We focus on the locus with the imposed condition $\mathrm{dim}(E_0\cap F_0)\equiv 0$ (mod $2$), but the analysis of the scheme structure with $\mathrm{dim}(E_0\cap F_0)\equiv 1$ (mod $2$) is similar. 

Let $G$ be the special orthogonal group $SO(V)=SO_{2n}$ in dimension $2n$. We denote $W_n^+$ as the set of signed permutations with even number of sign changes in $W_n=S_n\ltimes\{\pm1\}^n$. Here $S_n$ is the symmetric group. Often, $W_n$ is considered as the subgroup of $S_{2n}$ such that $w(\overline{i})=\overline{w(i)}$ for $w\in W_n$. Note that $\overline{a}$ indicates $-a$. 
For a signed permutation $v$ in $W_n^+$, one can define a rank function by
\[
r_v(a,\overline{b})=\#\{s<\overline{b}\;|\;v(s)>a\}
\] 
for $a\in\{0,\ldots, n-1\}$ and $b\in\{\overline{n-1},\ldots,\overline{0},0,\ldots n-1\}$, as in \cite[p. 10]{AIJK}. 

Let $Fl$ be an isotropic flag variety parametrizing flags 
\[
0\subset E_{n-1}\subset \cdots\subset E_1\subset E_0\subset V.
\]
of isotropic subbundles defined on $X$ with respect to the quadratic form.
Schubert varieties in the isotropic flag variety (or a degeneracy locus on a variety with isotropic bundles) are defined by
\begin{equation}\label{eqn:list}
\mathrm{dim}(E_b\cap F_a)\geq r_v(a,\overline{b}),
\end{equation}
for all $a$ and $0\leq b\leq n-1$. 
The rank function is useful to determine the Bruhat order on signed permutations. That is, for $u,v\in W_n^+$, we have $u\leq v$ in Bruhat order if and only if $r_u(a,\overline{b})\leq r_v(a,\overline{b})$ for all $a$ and $0\leq b\leq n-1$. 

One can minimize the list of conditions \eqref{eqn:list}, as some of them are redundant. Especially, we have a set of triples $S=\{(a_i,b_i, k_i)\}_i$ where the set of signed permutations $u$ with $r_u(a,\overline{b})\geq k_i$ has a unique minimum signed permutation $v$ in Bruhat order. This enables us to define the Schubert varieties associated with $v$ by
\[
\mathrm{dim}(E_{a_i}\cap F_{b_i})\geq k_i
\] 
for $(a_i,b_i,k_i)\in S$. 

 One example of this set is the essential set by Fulton \cite[Lemma 3.10]{Ful92}. In our situation, we choose a reduced list of conditions as follows. Given any $\mathbf{p}'=(p'_0,\ldots,p'_s)$ and $\mathbf{q}'=(q'_0,\ldots,q'_s)$, we define a triple $\tau:=\tau(\mathbf{p}',\mathbf{q}')$ of three sequences $\mathbf{p}=(p_0,\ldots,p_s)$, $\mathbf{q}=(q_0,\ldots,q_s)$ and $\mathbf{k}=(k_0,\ldots,k_s)$ of even length $s+1$ by 
 \begin{align}
\label{a}p_{2i}&=p'_i, &q_{2i}&=q'_i,\quad&\text{and}&\;&k_{2i}&=s+1-2i \\
\label{b} p_{2i+1}&=p'_i, &q_{2i+1}&=q'_{i+1}, \quad&\text{and}&\;&k_{2i+1}&=s+1-2i-1&\quad\text{if}\; p'_{i+1}-p'_i=1,\\
\label{c}p_{2i+1}&=p'_{i+1}, &q_{2i+1}&=q'_i, \quad&\text{and}&\;&k_{2i+1}&=s+1-2i-1&\quad\text{if}\; q'_{i+1}-q'_i=1
\end{align}
for $0\leq i\leq\lfloor s/2 \rfloor$. Moreover, if $s+1$ is odd, we add $p_{-1}=q_{-1}=0$ or remove the case when $(p'_0,q'_0)=(0,0)$ to have a triple of even length.

For example, let $\mathbf{p}'=(5,8,9,12,14)$ and $\mathbf{q}'=(1,2,6,10,11)$ with $s=4$. Since $s+1=5$ is odd, we set $p_{-1}=q_{-1}=0$. Then we have $\mathbf{p}=(0,5,8,8,8,9)$ and $\mathbf{q}=(0,1,1,2,6,6)$ with corresponding $k_i=s+1-i$ for $i=-1,0,\ldots s$.

For the triple $\tau$ arising this way, we can recover a {\it{vexillary signed permutation}} $w:=w(\tau)$ by the algorithm in \cite[\S 2]{AF20} after replacing $p_i$ and $q_i$ by $p_i+1$ and $q_i+1$ respectively. The corresponding partition $\lambda$ is given by $p_i+q_i$. 
 The vexillary element $w$ can be characterized as follows.

\begin{lemma}
The signed permutation $w$ is unique and minimal in Bruhat order such that
\[
r_w(q_i,\overline{p_i})=\#\{p<\overline{p_i}:w(p)>q_i\}\geq s+1-i \quad\text{for all}\;i,
\]
and its length is $\sum_i(p_i+q_i)$.
\end{lemma}
We remark that the length is equal to the codimension of the corresponding Schubert variety in the isotropic flag variety. Let us define a locus $V_{\mathbf{p},\mathbf{q}}$ by conditions
\[
\mathrm{dim}(E_{p_i}\cap F_{q_i})\geq s+1-i\quad\text{for all }i.
\]
Then the lemma implies that 
\[
V_{\mathbf{p},\mathbf{q}}=V_{\mathbf{p}',\mathbf{q}'}.
\] 
In fact, the procedure \eqref{a} covers the conditions when $i=j$ in \eqref{eqn:rk}. Furthermore, both \eqref{a} and \eqref{b} contains the case where $\mathrm{dim}(E_{p'_{i+1}}\cap F_{q'_{i}})\geq s+1-2i-1$, since 
\begin{align*}
r_w(q'_i,\overline{p'_{i+1}})&=r_w(q'_i,\overline{p'_i+1})\geq r_w(q'_i,\overline{p'_i})-1\geq s+1-2i-1.
\end{align*}
 Similarly we can see that the condition $\mathrm{dim}(E_{p'_i}\cap F_{q'_{i+1}})\geq s+1-2i-1$ is included in \eqref{a} and \eqref{c}. 
 Hence, the scheme structure of $V_{\mathbf{p}',\mathbf{q}'}=V_{\mathbf{p},\mathbf{q}}$ is induced by that of a vexillary Schubert variety associated to $w$ in the flag variety $Fl$ of type D.

\section{The two-pointed Prym-Brill-Noether loci}\label{sec3}

 In this section we discuss the Brill-Noether loci in the Prym variety $\mathscr{P}^\pm$ with special vanishings at two points, and investigate their classes in connective K-theory. 
 
 Let $C$ be a smooth algebraic curve of genus $g$ over an algebraically closed field $\mathbb{K}$ of characteristic not $2$. Let $\pi:\widetilde{C}\rightarrow C$ be an irreducible \'etale double covering. Classically, the covering $\pi$ determines a class $\epsilon$ of order $2$ in $\mathrm{Jac}(C)$ such that $\pi_*\mathcal{O}_{\widetilde{C}}=\mathcal{O}_C\oplus\epsilon$. Conversely, any non-trivial $2$-torsion point $\epsilon$ in $\mathrm{Jac}(C)$ defines an irreducible \'etale double covering $\pi:\widetilde{C}=\mathrm{Spec}(\mathcal{O}_C\oplus\epsilon)\rightarrow C$. So, we may represent the double covering $\pi$ by the pair $(C,\epsilon)$. 
 In particular, the \'etale double covering $\pi$ induces a norm map $\mathrm{Nm}:\mathrm{Pic}^{2g-2}(\widetilde{C})\rightarrow \mathrm{Pic}^{2g-2}(C)$. The scheme-theoretic inverse image $\mathrm{Nm}^{-1}(K_C)$ of the canonical class $K_C\in\mathrm{Pic}^{2g-2}(C)$ is given by the disjoint union 
 \[
\mathrm{Nm}^{-1}(K_C)\cong \mathscr{P}^+\sqcup\mathscr{P}^-,
 \]
of two connected irreducible components $\mathscr{P}^+=\{L\;|\;h^0(\widetilde{C},L)\equiv 0\;(\text{mod}\;2)\}$ and $\mathscr{P}^-=\{L\;|\;h^0(\widetilde{C},L)\equiv 1\;(\text{mod}\;2)\}$. Both varieties $\mathscr{P}^+$ and $\mathscr{P}^-$ are translates of the {Prym variety} $\mathscr{P}(C,\epsilon)=(\mathrm{Ker(Nm)})^0$ of dimension $g-1$, the connected component of the kernel of $\mathrm{Nm}:\mathrm{Jac}(\widetilde{C})\rightarrow \mathrm{Jac}(C)$ containing the origin. For more details about Prym varieties, see \cite[App. C]{ACGH} and \cite{Mum}. We denote by $\mathscr{P}^\pm$ either one of $\mathscr{P}^+$ or $\mathscr{P}^-$ and not the union of these varieties.

 We fix two points $P$ and $Q$ on the double cover $\widetilde{C}$ with $P-Q$ nontorsion. For $L\in \mathscr{P}^\pm$, let
\begin{align*}
\mathbf{a}'&=(0\leq a'_0< a'_1< \cdots<a'_r\leq 2g-2)\quad\text{and}\\
\mathbf{b}'&=(0\leq b'_0< b'_1< \cdots<b'_r\leq 2g-2)
\end{align*}
be strictly increasing sequences such that 
\[
\min\{a'_{i+1}-a'_i,b'_{i+1}-b'_i\}=1
\] for $i=0,\ldots, r-2$. 
For the above $\mathbf{a}'$ and $\mathbf{b}'$, the {\it two-pointed Prym-Brill-Noether loci} $V_{\mathbf{a}',\mathbf{b}'}^r(C,\epsilon,P,Q)\subset \mathscr{P}^\pm$  of line bundles associated to $(C,\epsilon, P, Q)$ is defined by 
 \begin{align}\label{def:prym}
 V_{\mathbf{a}',\mathbf{b}'}^r(C,\epsilon,P,Q):=\left\{L\in \right.& \mathrm{Nm}^{-1}(K_C)\;|\;h^0(\widetilde{C},L)\equiv r+1\;\text{(mod 2)},\\
&\left. h^0(\widetilde{C}, L(-b'_{j}Q'-a'_{i}P'))\geq r+1-j-i\;\text{for all $i,j$}\right\}\nonumber
 \end{align}
in $\mathrm{Pic}^{2g-2}(\widetilde{C}).$

Now, we provide the scheme structure of the two-pointed Prym-Brill-Noether loci as generalized version of the unpointed Prym-Brill-Noether loci \cite{DCP}, and analogous result to the pointed case in \cite{Tar}. 

Let $\mathcal{L}$ be the Poincar\'e line bundle on $\mathrm{Pic}^{2g-2}(\widetilde{C})\times \widetilde{C}$. We consider the double covering $1\times \pi: \mathrm{Pic}^{2g-2}(\widetilde{C})\times\widetilde{C}\rightarrow \mathrm{Pic}^{2g-2}(\widetilde{C})\times C$ induced by $\pi$, and set $\mathcal{E}=(1\times \pi)_*\mathcal{L}$. 
Then $\mathcal{E}|_{\mathscr{P}^\pm\times C}$, the restriction of $\mathcal{E}$ to $\mathscr{P}^\pm\times C$ is a rank $2$ vector bundle. The vector bundle $\mathcal{E}|_{\mathscr{P}^\pm\times C}$ is equipped with a nondegenerate quadratic form $\mathcal{Q}$. (See \cite[p.343]{Mum}, \cite[p.185]{Mum2}, and \cite[p.699]{DCP}.)

Let $q_1:\mathrm{Pic}^{2g-2}(\widetilde{C})\times C\rightarrow \mathrm{Pic}^{2g-2}(\widetilde{C})$ be the projection to the first factor and $q_2:\mathrm{Pic}^{2g-2}(\widetilde{C})\times C\rightarrow C$ the projection to the second factor. Let $\nu_1:\mathrm{Pic}^{2g-2}(\widetilde{C})\times \widetilde{C}\rightarrow \mathrm{Pic}^{2g-2}(\widetilde{C})$ be the first projection, and $\nu_2:\mathrm{Pic}^{2g-2}(\widetilde{C})\times \widetilde{C}\rightarrow \widetilde{C}$ be the second projection. These maps provide the commutative diagram, as follows.

\begin{equation*}\label{Diag}
\begin{tikzcd}[column sep=normal]
 &&\ar[lld, "\nu_1"', bend right=0]\mathrm{Pic}^{2g-2}(\widetilde{C})\times\widetilde{C}\ar[d,"1\times\pi"] \ar[rr,"\nu_2"]&&\widetilde{C}\ar[d,"\pi"]\; \\
\mathrm{Pic}^{2g-2}(\widetilde{C})  && \ar[ll,"q_1"] \mathrm{Pic}^{2g-2}(\widetilde{C})\times C\ar[rr, "q_2"] &&C\\
\end{tikzcd}
\end{equation*}

Let $D=\sum_{i=1}^N p_i$ denote a divisor on $C$ for sufficiently large enough $N$, where $p_i$'s are distinct such that $p_i\neq \pi(P),\pi(Q)$. Let $\widetilde{D}=\pi^*D$ and $\mathcal{L}(\pm \widetilde{D})=\mathcal{L}\otimes\nu_2^*\mathcal{O}_{\widetilde{C}}(\pm \widetilde{D})$. We define 
\[
\mathcal{V}:=(\nu_1)_*(\mathcal{L}(\widetilde{D})/\mathcal{L}(-\widetilde{D}))|_{\mathscr{P}^\pm}=(q_1)_*(\mathcal{E}(D)/\mathcal{E}(-D))|_{\mathscr{P}^\pm},
\]
where $\mathcal{E}(\pm D)=\mathcal{E}\otimes q_2^*(\mathcal{O}_C(\pm D))$. Let $n=2N=2\mathrm{deg}(D)$. The vector bundle $\mathcal{V}$ has rank $2n$ and is endowed with a nondegenerate quadratic form $\mathfrak{q}$ induced by the form $\mathcal{Q}$ with values in $\mathcal{O}_{\mathscr{P}^\pm}$.
For $0\leq i\leq r$, we let
\[
\mathcal{W}_{a_i'}:=(\nu_1)_*(\mathcal{L}(\widetilde{D}-a'_{i}P))|_{\mathscr{P}^\pm}\quad\text{and}\quad\mathcal{U}_{b_i'}:=(\nu_1)_*(\mathcal{L}(-b'_{i}Q)/\mathcal{L}(-\widetilde{D}))|_{\mathscr{P}^\pm}.
\]
Then $\mathcal{W}_{a_i'}$ and $\mathcal{U}_{b_i'}$ are subbundles of $\mathcal{V}$ and isotropic with respect to the form $\mathfrak{q}$, with their ranks
\[
\mathrm{rk}(\mathcal{W}_{a_i'})=n-a'_{i}\quad\text{and}\quad\mathrm{rk}(\mathcal{U}_{b_i'})=n-b'_{i},
\]
via the Riemann-Roch theorem. We thus have natural flags
\[
\mathcal{W}_{a_r'}\subset \mathcal{W}_{a_{r-1}'}\subset\cdots\subset \mathcal{W}_{a_0'}\subset \mathcal{V}
\quad\text{and}\quad
\mathcal{U}_{b_r'}\subset \mathcal{U}_{b_{r-1}'}\subset\cdots\subset \mathcal{U}_{b_0'}\subset \mathcal{V}
\]
of vector bundles on $\mathscr{P}^\pm$. 

We take $L\in \mathscr{P}^{\pm}$, and let $V:=H^0(\widetilde{C},L(\widetilde{D})/L(-\widetilde{D}))$ be a vector space with a nondegenerate form induced by $\mathfrak{q}$. We consider subspaces $W_{a_i'}:=H^0(\widetilde{C},L(\widetilde{D}-a'_i P'))$ and $U_{b_i'}:=H^0(\widetilde{C},L(-b'_iQ')/L(-\widetilde{D}))$ of $V:=H^0(\widetilde{C},L(\widetilde{D})/L(-\widetilde{D}))$. Then $W_{a_i'}$'s and $U_{b_j'}$ are all isotropic to the nondegenerate quadratic form. Furthermore, through Mumford's construction \cite[p. 183]{Mum2}, we establish $H^0(\widetilde{C},L(-b_j'Q-a_i'P))$ as the intersection of two isotropic subspaces $W_{a_i'}$ and $U_{b_j'}$. That is, 
\[
H^0(\widetilde{C},L(-b_j'Q-a_i'P))=W_{a_i'}\cap U_{b_j'}\subset V.
\] 
 In particular, we can globalize this construction over the Prym varieties $\mathscr{P}^\pm$, so that
$V_{\mathbf{a}',\mathbf{b}'}^r(C,\epsilon,P,Q)$ can be defined set-theoretically by the condition 
 \begin{equation}\label{cond}
 \mathrm{dim}(\mathcal{W}_{a_i'}\cap\mathcal{U}_{b_j'})\geq r+1-j-i,
 \end{equation}
 the fiberwise intersection of two isotropic subbunddles $\mathcal{W}_{a_i'}$ and $\mathcal{U}_{b_j'}$ in dimension at least $r+1-j-i$ for all $i,j$.

 Combined with the triple $\tau(\mathbf{a}',\mathbf{b}')=(\mathbf{a},\mathbf{b},\mathbf{k})$ from \S\ref{sec2}, it is equivalent to say that the two-pointed Prym-Brill Noether loci $V_{\mathbf{a},\mathbf{b}}^r(C,\epsilon,P,Q)=V_{\mathbf{a}',\mathbf{b}'}^r(C,\epsilon,P,Q)$ is the locus of line bundles $L$ in $\mathscr{P}^\pm$ such that  
 \[
 h^0(\widetilde{C},L(-b_iQ-a_iP))=\mathrm{dim}(\mathcal{W}_{a_i}\cap\mathcal{U}_{b_i})\geq r+1-i\quad\text{ for $i=0,\ldots,r$}.
 \]

 The scheme structure of this presentation arises in a natural way from the scheme structure inherent in a Schubert variety associated to a vexillary element $w(\tau)$ in the isotropic flag variety $Fl$ of type D, where $\tau=\tau(\mathbf{a}',\mathbf{b}')$. We emphasize that the scheme structure should be taken by the closure of the locus where the equality holds, as described in \cite[\S4]{And19}. 
 Thus, $V_{\mathbf{a},\mathbf{b}}^r(C,\epsilon,P,Q)$ has the expected dimension $g-1-|\mathbf{a}+\mathbf{b}|$ where $|\mathbf{a}+\mathbf{b}|:=\sum_i(a_i+b_i)$. Moreover, if $V_{\mathbf{a}',\mathbf{b}'}^r(C,\epsilon,P,Q)$ has pure codimension $|\mathbf{a}+\mathbf{b}|$, it is Cohen-Macaulay by \cite{Ra}, cf. \cite[Proposition 2 (2)]{DCP}.  
 
 Now, we provide formulas for the K-theory classes of such two-pointed Prym-Brill-Noether loci $V_{\mathbf{a},\mathbf{b}}^r(C,\epsilon,P,Q)$. 
 
For an irreducible variety $X$, we denote by $CK^*(X)$ the connective K-cohomology which is a graded algebra over $\mathbb{Z}[\beta]$ where $\beta$ is in degree $-1$. Especially, $CK^*(X)$ interpolates K-cohomology $K^\circ(X)$ at $\beta=-1$ and the Chow ring $A^*(X)$ at $\beta=0$. We refer the reader to \cite[Appendix A]{And19}, \cite[\S 1.1]{ACT} for the relevant facts of the connective K-theory.

 Let $\lambda$ be the partition defined by $\lambda_{i}=a_{r+1-i}+b_{r+1-i}$. For $i=1,\ldots, r+1$, let $c(i)=c^K((\mathcal{W}_{r+1-j})_{\mathscr{P}^\pm}^\vee)$ be the K-theoretic Chern class of $(\mathcal{W}_{r+1-j})_{\mathscr{P}^\pm}^\vee$ and
\begin{equation*}
e_m(i) =
\begin{cases}
(-1)^{\mathrm{dim}(\mathcal{U}\cap\mathcal{W})}\gamma(\mathcal{W}_{\mathscr{P}^\pm},\mathcal{U}_{\mathscr{P}^\pm}) c^K_{\lambda_{i}}(\mathcal{W}_{\mathscr{P}^\pm}/(\mathcal{W}_{r+1-i})_{\mathscr{P}^\pm}) & \text{if $m=\lambda_{i}$}\\
0& \text{otherwise}
\end{cases}
\end{equation*}
as the specialized {\it Euler classes}. See \cite[Appendix B]{And19} for the Euler classes in details. Since the bundles $\mathcal{U}_i$ have trivial Chern classes for sufficiently positive $\widetilde{D}$ as in \cite[Lemma 5.1]{DCP}, we have the following theorem by \cite[Theorem 4]{And19}.

Let $d(i)=c(i)+\sigma(i)e(i)$ for $\sigma(i)=(-1)^i$, $i=1,\ldots, r+1$. Let $T_i$ be the {\it raising operator} raising the index of $c(i)$ by one, and let $R_{ij}=T_i/T_j$. Let $\delta_i$ denote the operator acting by sending $\sigma(i)$ to $0$. So, the action of $\delta_i$ replaces $d(i)$ by $c(i)$. We write $\widetilde{T}_i$ for $\delta_iT_i$.

\begin{theorem}\label{thm:ck}
Let $\ell_\circ:=\ell(\lambda)$ be the number of non-zero components of $\lambda$. The dimension of $V_{\mathbf{a},\mathbf{b}}^r(C,\epsilon,P,Q)$ is at least $g-1-|\mathbf{a}+\mathbf{b}|$. If $\mathrm{dim}(V_{\mathbf{a},\mathbf{b}}^r(C,\epsilon,P,Q))= g-1-|\mathbf{a}+\mathbf{b}|$, then we have
\begin{align*}
\left[V_{\mathbf{a},\mathbf{b}}^r(C,\epsilon,P,Q)\right]&=Pf_\lambda(d(1),\ldots,d(\ell_\circ);\beta)\\
&:=Pf(M)
\end{align*}
in $CK^*(\mathscr{P}^\pm)[1/2]$. Here $M$ is the skew-symmetric matrix with entries 
\begin{align*}
m_{i,j}&=\dfrac{1-\delta_i\delta_jR_{ij}}{1+\delta_i\delta_j(R_{ij}-\beta T_i)}\cdot\dfrac{(1-\beta \widetilde{T}_i)^{{\ell_\circ}-i-\lambda_i+1}}{2-\beta\widetilde{T}_i}\cdot\dfrac{(1-\beta\widetilde{T}_j)^{{\ell_\circ}-j-\lambda_j+1}}{2-\beta\widetilde{T}_j}\\
&\cdot (c_{\lambda_i}(i)-(-1)^{\ell_\circ}e_{\lambda_i}(i))\cdot(c_{\lambda_j}(j)+(-1)^{\ell_\circ}e_{\lambda_j}(j).
\end{align*}
If $\ell_\circ$ is odd, we augment the matrix $M$ by
\begin{align*}
\label{aug}
m_{0j}&=(1-\beta\widetilde{T}_j)^{{\ell_\circ}-j-\lambda_j+1}(2-\beta\widetilde{T}_j)^{-1}\cdot(c_{\lambda_j}(j)+ e_{\lambda_j}(j)).
\end{align*}
\end{theorem}
By the property in the connective K-theory, we have the class in the chow ring $A^*(\mathscr{P}^\pm)$ at $\beta=0$, and K-theory class in $K^\circ(\mathscr{P}^\pm)$ at $\beta=-1$.

 In addition, we have Corollary \ref{cor33} from Theorem \ref{thm:ck} at $\beta=0$ and \cite[Lemma 5.2]{DCP} that the Chern classes $c(\mathcal{W}_i^\vee)=e^{2\xi}$ for sufficiently large $\widetilde{D}>>0$, which is independent of the choice of $i$.

In Corollary \ref{cor33}, when we take $\mathbf{a}$ and $\mathbf{b}$ constructed by $\mathbf{a}'=(0,\ldots,r)$ and $\mathbf{b}'=(0,\ldots, r)$ with \eqref{a} and \eqref{c}, the above corollary gives the result \cite[Theorem 9]{DCP}. Moreover, if we take $P=Q$, then we obtain the class of the pointed Prym-Brill-Noether loci with special vanishing orders $\mathbf{a}+\mathbf{b}=(a_i+b_i)_i$ at $P$.

\section{The coupled Prym-Petri map}\label{sec4}
In this section we modify the ideas in \cite{Tar} which consider the pointed case and adjust the classical Brill-Noether case for two points in \cite{Pflu}.

Let $C$ be a smooth curve of genus $g$. Let $\pi:\widetilde{C}\rightarrow C$ be an irreducible \'etale double covering associated to a non-trivial $2$-torsion point $\epsilon$ in $\mathrm{Jac}(C)$. Let $\iota:\widetilde{C}\rightarrow \widetilde{C}$ be the involution of $\widetilde{C}$. The involution acts on the space $H^0(\widetilde{C},K_{\widetilde{C}})$ so that one has a decomposition 
\[
H^0(\widetilde{C},K_{\widetilde{C}})\cong H^0(\widetilde{C},K_{\widetilde{C}})^+\oplus H^0(\widetilde{C},K_{\widetilde{C}})^-,
\]
where $H^0(\widetilde{C},K_{\widetilde{C}})^+$ is the space of invariants and $H^0(\widetilde{C},K_{\widetilde{C}})^-$ is the anti-invariant sections. Since $K_{\widetilde{C}}=\pi^*K_{C}$, and by the push-pull formula, the space $H^0(\widetilde{C},K_{\widetilde{C}})$ of differential forms splits into a direct sum 
\[
H^0(\widetilde{C},K_{\widetilde{C}})\cong H^0(C,K_{C})\oplus H^0(C,K_{C}\otimes\epsilon).
\]

Given the norm map $\mathrm{Nm}:\mathrm{Pic}(\widetilde{C})\rightarrow\mathrm{Pic}(C)$ induced by the covering $\pi$ as before, we have $\pi^*\circ \mathrm{Nm}=id_{\mathrm{Pic}(\widetilde{C})}\otimes\epsilon$ and $\mathrm{Nm}\circ\pi^*=2 id_{\mathrm{Pic}(C)}$, see \cite[(3.1.1-2)]{Sho}.

Let $P,Q\in\widetilde{C}$. Given $\bold{a}'=(0\leq a'_0<a'_1<\cdots<a'_r)$ and $\bold{b}'=(0\leq b'_0<b'_1<\cdots<b'_r)$ such that $\mathrm{min}\{a'_{i+1}-a'_i,b'_{i+1}-b'_i\}=1$ for $i=0,\ldots, r-2$, we have $\mathbf{a}=(a_0,\ldots,a_r)$ and $\mathbf{b}=(b_0,\ldots,b_r)$ of even length (by putting $a_{-1}=b_{-1}=0$ if necessary) from \eqref{a},\eqref{b},\eqref{c}, satisfying
 $h^0(\widetilde{C},L(-a_iP-b_iQ))=r+1-i$ for all $i$, $L\in V_{\mathbf{a},\mathbf{b}}^r(C,\epsilon,P,Q)$. We mainly focus on the case where $r+1$ is even for the rest of this paper, since the one where $r+1$ is odd follows similarly.
 
 We define 
\begin{equation}\label{aa}
T_{P,Q}^L(\mathbf{a},\mathbf{b})=\sum_{0\leq i\leq r}H^0(\widetilde{C},L(-a_iP-b_iQ))\otimes H^0(\widetilde{C},K_{\widetilde{C}}\otimes L^\vee(a_iP+b_iQ)).
\end{equation}
We will use $T_{P,Q}^L$ for $T_{P,Q}^L(\mathbf{a},\mathbf{b})$ as an abbreviation. It is worthwhile to note that $H^0(\widetilde{C},L(-a_{i+1}P-b_{i+1}Q))\subset H^0(\widetilde{C},L(-a_iP-b_iQ))$ such that there might be a nonempty intersection between terms in \eqref{aa}. The {\it{coupled Prym-Petri map}} is given by the composition $p\circ\mu$
\[
\mu_{P,Q}^L:T_{P,Q}^L\xrightarrow{\mu} H^0(\widetilde{C},K_{\widetilde{C}})\xrightarrow{p}H^0(C,K_C\otimes \epsilon).
\]

To be more specific, one may view the coupled Prym-Petri map as follows. 
We take sections
\begin{align*}
\sigma_{i}&\in H^0(\widetilde{C},L(-a_iP-b_iQ))\backslash \left(H^0(\widetilde{C},L(-a_{i+1}P-b_{i+1}Q))\right)\;\text{for}\;0\leq i\leq r-1,\;\text{and}\\
\sigma_{r}&\in H^0(\widetilde{C},L(-a_rP-b_rQ)).
\end{align*}
Then the composition $p\circ \mu$
\[
\displaystyle\bigoplus_{i=0}^r\langle \sigma_{i}\rangle\otimes H^0(\widetilde{C},K_{\widetilde{C}}\otimes L^\vee(a_iP+b_iQ))\xrightarrow{\mu}H^0(\widetilde{C},K_{\widetilde{C}})\xrightarrow{p}H^0(C,K_C\otimes\epsilon)
\] 
can be defined by sending $\sigma\otimes \tau$ to $1/2(\sigma.\tau-\iota(\sigma).\iota(\tau))$ where $\iota$ is the involution for the double cover $\widetilde{C}$.
In particular, given the inclusions
\[
\iota^*H^0(\widetilde{C},L(-a_iP-b_iQ))\hookrightarrow \iota^*H^0(\widetilde{C},L)\xrightarrow{\cong}H^0(\widetilde{C},K_{\widetilde{C}}\otimes L^\vee)\hookrightarrow H^0(\widetilde{C},K_{\widetilde{C}}\otimes L^\vee(a_iP+b_iQ))
\]
for each $i$, we may call the composition $p\circ \mu$ restricted to a map

\[
\overline{\mu}:\displaystyle\bigoplus_{i}\langle \sigma_{i}\rangle \otimes H^0(\widetilde{C},K_{\widetilde{C}}\otimes L^\vee(a_iP+b_iQ))/\iota^*H^0(\widetilde{C},L(-a_iP-b_iQ))\rightarrow H^0(C,K_C\otimes \epsilon)
\]
the coupled Prym-Petri map $\overline{\mu}$ for $L$.

We remark that the coupled Prym-Petri map recovers the Prym-Petri map in \cite{Wel85} with specialization $\bold{a}'=(0,1,\ldots,r)$ and $\mathbf{b}'=(0,1,\ldots,r)$.

Now, we provide information regarding the tangent space of the two-pointed Prym-Brill Noether loci with the coupled Prym-Petri map.
Let us consider the natural maps $\xi_{i}$ 
\begin{align*}
\xi_{i}:&H^1(\widetilde{C},\mathscr{O}_{\widetilde{C}})\rightarrow \mathrm{Hom}(H^0(\widetilde{C},L(-a_iP-b_iQ)),H^1(\widetilde{C},L(-a_iP-b_iQ))),
\end{align*}
as the transpose of 
$$\mu_{i}:H^0(\widetilde{C},K_{\widetilde{C}}\otimes L^\vee(a_iP+b_iQ))\otimes H^0(\widetilde{C},L(-a_iP-b_iQ))\rightarrow H^0(\widetilde{C},K_{\widetilde{C}})$$ given by $\xi_{i}(v)(\sigma)=v\cup\sigma$.  
More generally, we define $\xi_{i}$ inductively on $0\leq i\leq r-1$  by
\begin{align*}
\xi_{i}:\mathrm{ker}(\xi_{i+1})\rightarrow \mathrm{Hom}(&H^0(\widetilde{C},L(-a_iP-b_iQ))/H^0(\widetilde{C},L(-a_{i+1}P-b_{i+1}Q)),\\
&H^1(\widetilde{C},L(-a_iP-b_iQ))),
\end{align*}
which comes from the restriction of the natural maps
\[
H^1(\widetilde{C},\mathscr{O}_{\widetilde{C}})\rightarrow \mathrm{Hom}(H^0(\widetilde{C},L(-a_iP-b_iQ)),H^1(\widetilde{C},L(-a_iP-b_iQ))).
\]
In fact, the locus $V_{\mathbf{a},\mathbf{b}}^r(C,\epsilon,P,Q)$ contains the Zariski open subset
\begin{align*}
V_{\mathbf{a},\mathbf{b}}^r(C,\epsilon,P,Q)^\circ&:=W_{2g-2}^{\tau}(\widetilde{C},P,Q)^\circ\cap \mathscr{P}^+\quad\text{if $r$ is odd},\\
V_{\mathbf{a},\mathbf{b}}^r(C,\epsilon,P,Q)^\circ&:=W_{2g-2}^{\tau}(\widetilde{C},P,Q)^\circ\cap \mathscr{P}^-\quad\text{if $r$ is even},
\end{align*}
where
\[
W_{2g-2}^{\tau}(\widetilde{C},P,Q)^\circ:=\{L\in\mathrm{Pic}^{2g-2}(\widetilde{C}):h^0(\widetilde{C},L(-a_iP-b_iQ))=r+1-i\quad\text{for all $i$}\}.
\]
Then the scheme structure of $V_{\mathbf{a},\mathbf{b}}^r(C,\epsilon,P,Q)$ along $V_{\mathbf{a},\mathbf{b}}^r(C,\epsilon,P,Q)^\circ$ can be seen as the scheme structure on $W_{2g-2}^{\tau}(\widetilde{C},P,Q)^\circ\cap\mathscr{P}^\pm$, as shown in \cite[Prop. 4(1)]{DCP}, \cite[\S 2]{Tar}.

\begin{lemma}\label{lem:tan}
The tangent space $T_L(W_{2g-2}^{\tau}(\widetilde{C}, P,Q)^\circ)$ is isomorphic to $\mathrm{ker}(\xi_{0})$.
\end{lemma}

\begin{proof}
Similar arguments used for the classical Brill-Noether case with one vanishing point \cite[(3.1)]{CHT} hold for the proof of Lemma \ref{lem:tan}.

Let $L$ be a line bundle on $\widetilde{C}$ in $W_{2g-2}^{\tau}(\widetilde{C}, P,Q)^\circ$. Let us take an affine covering $\{U_k\}$ of $\widetilde{C}$ such that the line bundle $L$ trivializes over it. We denote by $\{f_{kl}\}$ the transition functions of the covering. The trivial infinitesimal deformation $\widetilde{C}_\epsilon$ of $\widetilde{C}$ has a covering $\{U_{k\epsilon}\}$ where $U_{k\epsilon}$ is the deformation of $U_k$, and the deformation $L_\epsilon$ of $L$ trivializes by transition functions of the form $f_{kl}(1+\epsilon g_{kl})$. Since such family $\{g_{kl}\}$ can be considered as elements of $H^1(\widetilde{C},\mathscr{O}_{\widetilde{C}})$, one can correspond the deformation $L_\epsilon$ to an element of $H^1(\widetilde{C},\mathscr{O}_{\widetilde{C}})$.  
 
 Given these setting, we let $\mathcal{S}$ denote a section of $L$ vanishing with orders $a_i$ at $P$ and $b_i$ at $Q$. The section $\mathcal{S}$ can be deformed to a section of $L_\epsilon$ preserving the orders of vanishing on $P$ and $Q$. To be specific, if $P$ and $Q$ are in some $U_k$, we can have a section $\mathcal{S}_k'$ defined on $U_k$ with vanishing to the orders $a_i$ on $P$ and $b_i$ on $Q$ such that 
\[
f_{kl}(1+\epsilon g_{kl})(\mathcal{S}_k+\epsilon \mathcal{S}_k')=\mathcal{S}_l+\epsilon \mathcal{S}_l'
\]
by taking 
\[
g_{kl}\mathcal{S}_l=\mathcal{S}_l'-f_{kl}\mathcal{S}_k'.
\]
This, in fact, implies that $g\mathcal{S}$ vanishes in $H^1(\widetilde{C},L(-a_iP-b_iQ))$, which gives the family $\{\mathcal{S}_k+\epsilon\mathcal{S}_k'\}$ as a section of $L_\epsilon$. An infinitesimal deformation of $H^1(\widetilde{C}, \mathscr{O}_{\widetilde{C}})$ lies in $\mathrm{ker}(\xi_{i})$ for all $i$ if and only if it belongs to the tangent space $T_L(W_{2g-2}^{\tau}(\widetilde{C}, P,Q)^\circ)$. Hence, we obtain the lemma, since $\mathrm{ker}(\xi_{0})$ is the subspace of $H^1(\widetilde{C},\mathscr{O}_{\widetilde{C}})$ included in all $\mathrm{ker}(\xi_{i})$.
\end{proof}

According to \cite[pg.673]{Wel85}, the tangent space $T_L(\mathscr{P}^\pm)$ of the Prym varieties at $L$ is given by
\[
T_L(\mathscr{P}^\pm)=H^0(C,K_C\otimes\epsilon)^\vee\xhookrightarrow{p^t} H^0(\widetilde{C},K_{\widetilde{C}})^\vee.
\]
Here $p^t$ is the transpose of the projection map $p$. So, with Lemma \ref{lem:tan}, we get the tangent space
\begin{equation}\label{eqn:tan}
T_L(V_{\mathbf{a},\mathbf{b}}^r(C,\epsilon,P,Q)^\circ)=\mathrm{ker}(\xi_{0})\cap H^0(C,K_C\otimes\epsilon)^\vee.
\end{equation}
Since $\xi_{0}$ is the transpose of $\mu$, $\mathrm{ker}(\xi_{0})=\mathrm{Im}(\mu)^\perp$. Then the righthand side of \eqref{eqn:tan} becomes $\mathrm{Im}(\mu)^\perp\cap H^0(\widetilde{C},K_{\widetilde{C}}\otimes\epsilon)^\vee$ so that 
\[
T_L(V_{\mathbf{a},\mathbf{b}}^r(C,\epsilon,P,Q)^\circ)=\mathrm{Im}(\overline{\mu})^\perp.
\]
In fact, by the Riemann-Roch formula, we have
\[
h^0(\widetilde{C},K_{\widetilde{C}}\otimes L^\vee(a_iP+b_iQ))=r+1-i+a_i+b_i\quad\text{for}\; 0\leq i\leq r.
\]
We know that 
\[
\mathrm{dim}\; H^0(C,K_C\otimes \epsilon)=\mathrm{dim}\;\mathrm{Im}(\overline{\mu})+\mathrm{dim}\;\mathrm{ker}(\overline{\mu}^t).
\]
We denote by $|\mathbf{a}|$ the sum of all $a_i$s and $|\mathbf{b}|$ the sum of $b_i$s. Since the dimension of the domain of $\overline{\mu}$ is $|\mathbf{a}|+|\mathbf{b}|$ and $\mathrm{dim}(H^0(C,K_C\otimes\epsilon))=g-1$, we can deduce the dimension of the tangent space $T_L(V_{\mathbf{a},\mathbf{b}}^r(C,\epsilon,P,Q)^\circ)$ as
\[
\mathrm{dim}(T_L(V_{\mathbf{a},\mathbf{b}}^r(C,\epsilon,P,Q)^\circ))=g-1-|\mathbf{a}|-|\mathbf{b}|+\mathrm{dim}\;\mathrm{ker}(\overline{\mu}).
\]

\begin{proposition}\label{prop:4.2}
Let $L\in V_{\mathbf{a},\mathbf{b}}^r(C,\epsilon,P,Q)^\circ$. Then we have
\[
\mathrm{dim}_L(V_{\mathbf{a},\mathbf{b}}^r(C,\epsilon,P,Q))\leq g-1-|\mathbf{a}|-|\mathbf{b}|+\mathrm{dim}\;\mathrm{ker}(\overline{\mu}).
\]
\end{proposition}

\begin{definition}
For general points $P$ and $Q$ in the \'{e}tale double cover $\widetilde{C}$, a quadruple $(C,\epsilon,P,Q)$ satisfies {\it the coupled Prym-Petri condition} if $\overline{\mu}$ is injective for all $L\in V_{\mathbf{a},\mathbf{b}}^r(C,\epsilon,P,Q)$ such that $h^0(\widetilde{C},L(-a_iP-b_iQ))=r+1-i$ for all $i$.
\end{definition}
Proposition \ref{prop:4.2} and the inequality on the dimension of the two-pointed Prym-Brill-Noether loci in Corollary \ref{cor:3.3} give rise to the following proposition.

\begin{proposition}\label{prop:4.4}
Suppose that $(C,\epsilon, P,Q)$ satisfies the coupled Prym-Petri condition, then $V_{\mathbf{a},\mathbf{b}}^r(C,\epsilon,P,Q)$ is either empty, or smooth of dimension $g-1-|\mathbf{a}|-|\mathbf{b}|$ at $L\in V_{\mathbf{a},\mathbf{b}}^r(C,\epsilon,P,Q)$ such that  $h^0(\widetilde{C},L(-a_iP-b_iQ))= r+1-i\;\text{for all $i$}$.
\end{proposition}

\section{General case}\label{sec5}

The ultimate goal of this section \S\ref{sec5} is to prove Theorem \ref{thm:main}. We adeptly draw upon and refine concepts from the unpointed case as outlined in \cite{Wel85} and the pointed case in \cite{Tar}. Moreover, we investigate the adaptation of the classical Brill-Noether case as explored in \cite{EH83} and pointed Brill-Noether case \cite{CHT}, in order to establish the proof of Theorem \ref{thm:main} in the context of two-pointed Brill-Noether case.

Let us first consider a quadruple $(\mathscr{C},\epsilon, P, Q)$, which arises from the geometric generic fiber of a family, as described below.

We denote by $T=\mathrm{Spec}(\mathcal{O})$ the spectrum of a discrete valuation ring $\mathcal{O}$ with parameter $t$. It has a special point at $0$ and a generic point at $\eta$. We should note that $T$ has trivial Picard group, which will be used later. Let $\phi:\mathscr{C}\rightarrow T$ indicate a flat projective family with a smooth surface $\mathscr{C}$. This family satisfies the following conditions: (i) the generic fiber $\mathscr{C}_\eta$ is smooth and geometrically irreducible, and (ii) the special fiber $\mathscr{C}_0$ is a reduced curve of arithmetic genus $g$. It consists of a sequence of smooth components, which can be either rational components or elliptic curves denoted as $E_1,\ldots, E_g$, as in Fig. \ref{fig:1}. The special fiber only has ordinary double points as singularities. It is required further that for $1\leq i\leq g$, each $E_i$ intersect the other components of $\mathscr{C}_0$ at the two points $P_i$ and $Q_i$. Here $P_i-Q_i\in \mathrm{Pic}^0(E_i)$ is not torsion.

\begin{figure}[htb]

\begin{tikzpicture}[scale=0.764]\draw [draw=black,ultra thin] (-5.5,-0.3) -- (-4.5,0.2);  
\node at (-4,0) {\ldots};

\draw [draw=black,ultra thin] (-2.5,-0.3) -- (-3.5,0.2);  
\draw [draw=black,ultra thin] (-3,-0.3) -- (-2,0.2);  
\draw (-1.27,0.2) ellipse (1cm and 0.7cm)
node[left=13, below=26, label={0:$E_{1}''$}]{};
\node at (-2.25, 0.07) {$\bullet$};
\node at (-2.3, 1.5) {$P_1''$};
\draw [draw=black,ultra thin] (0,-0.3) -- (1,0.2);  
\draw [draw=black,ultra thin] (0.5,-0.3) -- (-.5,0.2);  
\node at (1.6,0) {\ldots}; 
\node at (-.3, 0.1) {$\bullet$};
\node at (-.25, 1.5) {$Q_1''$};

\draw [draw=black,ultra thin] (3.15,-0.3) -- (2.15,0.2);  
\draw [draw=black,ultra thin] (2.65,-0.3) -- (3.65,0.2);  
\draw (4.35,0.2) ellipse (1cm and 0.7cm)
node[left=13, below=26, label={0:$E_{g-1}''$}]{};
\node at (3.4, 0.07) {$\bullet$};
\node at (3.45, 1.5) {$P_{g-1}''$};
\draw [draw=black,ultra thin] (5.65,-0.3) -- (6.65,0.2);  
\draw [draw=black,ultra thin] (6.15,-0.3) -- (5.15,0.2);  
\node at (7.25,0) {\ldots}; 
\node at (5.35, 0.1) {$\bullet$};
\node at (5.34, 1.5) {$Q_{g-1}''$};

 \begin{scope}[xshift=320, yshift=30, xscale=1, yscale=1]
    \begin{knot}[clip width=10, clip radius=15pt, consider self intersections,
    end tolerance=3pt,xscale=0.97, yscale=1,xshift=-10]
       \strand[thin] (0.3,2.7) 
          to[out=down, in=down] (-2,2.73)
            to[out=up, in=left] (-1.6,3.2)
          to[out=right, in=up] (0.35,-0.6)
           to[out=down, in=right]  (-0.8,-1.6)
          to[out=left, in=down] (-2.07,-.5)
           to[out=up, in=left]  (-0.2,3.2)
                     to[out=right, in=up]  (0.3,2.7);

    \end{knot}
    
\draw [draw=black,ultra thin] (-2.5,2.2) -- (-3.5,2.7);  
\draw [draw=black,ultra thin] (-3,2.2) -- (-2,2.7);  

\draw [draw=black,ultra thin] (-2.5,-1.3) -- (-3.5,-0.8);  
\draw [draw=black,ultra thin] (-3,-1.3) -- (-2,-0.8);  

 \draw [draw=black,ultra thin] (0.7,-1.3) -- (-0.3,-0.8);  
\draw [draw=black,ultra thin] (0.2,-1.3) -- (1.2,-0.8);   

\draw [draw=black,ultra thin] (0.7,2.2) -- (-0.3,2.7);  
\draw [draw=black,ultra thin] (0.2,2.2) -- (1.2,2.7);  

\node at (-3, .45) {$P_{g}''$};
\node at (.6, .45) {$Q_{g}''$};

\node at (-2.25, 2.58) {$\bullet$};
\node at (-.058, 2.58) {$\bullet$};
\node at (-2.25, -0.93) {$\bullet$};
\node at (-.058, -0.93) {$\bullet$};
\node at (-1.87, 4) {$P_{g}'$};
\node at (0, 4) {$Q_{g}'$};
\node at (1.6, 4) {};
  \node[left=37, below=44, label={0:$\widetilde{E}_{g}$}]{};

\end{scope}

\begin{scope}[xshift=0, yshift=100 ]
\draw [draw=black,ultra thin] (-5.5,-0.3) -- (-4.5,0.2);  
\node at (-4,0) {\ldots}; 
\node at (-5.3, 1.5) {$P_0$};
\node at (-5.3, -0.2) {$\bullet$};
\node at (-4.5, 1.5) {$Q_0$};
\node at (-4.7, .1) {$\bullet$};

\draw [draw=black,ultra thin] (-2.5,-0.3) -- (-3.5,0.2);  
\draw [draw=black,ultra thin] (-3,-0.3) -- (-2,0.2);  
\draw (-1.27,0.2) ellipse (1cm and 0.7cm)
node[left=13, below=26, label={0:$E_{1}'$}]{};
\node at (-2.25, 0.07) {$\bullet$};
\node at (-2.3, 1.5) {$P_1'$};
\draw [draw=black,ultra thin] (0,-0.3) -- (1,0.2);  
\draw [draw=black,ultra thin] (0.5,-0.3) -- (-.5,0.2);  
\node at (1.6,0) {\ldots}; 
\node at (-.3, 0.1) {$\bullet$};
\node at (-.25, 1.5) {$Q_1'$};

\draw [draw=black,ultra thin] (3.15,-0.3) -- (2.15,0.2);  
\draw [draw=black,ultra thin] (2.65,-0.3) -- (3.65,0.2);  
\draw (4.35,0.2) ellipse (1cm and 0.7cm)
node[left=13, below=26, label={0:$E_{g-1}'$}]{};
\node at (3.4, 0.07) {$\bullet$};
\node at (3.45, 1.5) {$P_{g-1}'$};
\draw [draw=black,ultra thin] (5.65,-0.3) -- (6.65,0.2);  
\draw [draw=black,ultra thin] (6.15,-0.3) -- (5.15,0.2);  
\node at (7.25,0) {\ldots}; 
\node at (5.35, 0.1) {$\bullet$};
\node at (5.32, 1.5) {$Q_{g-1}'$};

\end{scope}

\end{tikzpicture}

$\DownArrow[20pt][>=latex,black, ultra thick] $

\begin{tikzpicture}[scale=0.8]
\draw [draw=black,ultra thin] (-5.5,-0.3) -- (-4.5,0.2);  
\node at (-4,0) {\ldots};

\draw [draw=black,ultra thin] (-2.5,-0.3) -- (-3.5,0.2);  
\draw [draw=black,ultra thin] (-3,-0.3) -- (-2,0.2);  
\draw (-1.27,0.2) ellipse (1cm and 0.7cm)
node[left=13, below=26, label={0:$E_{1}$}]{};
\node at (-2.25, 0.07) {$\bullet$};
\node at (-2.3, 1.5) {$P_1$};
\draw [draw=black,ultra thin] (0,-0.3) -- (1,0.2);  
\draw [draw=black,ultra thin] (0.5,-0.3) -- (-.5,0.2);  
\node at (1.6,0) {\ldots}; 
\node at (-.3, 0.1) {$\bullet$};
\node at (-.25, 1.5) {$Q_1$};

\draw [draw=black,ultra thin] (3,-0.3) -- (2,0.2);  
\draw [draw=black,ultra thin] (2.5,-0.3) -- (3.5,0.2);  
\draw (4.2,0.2) ellipse (1cm and 0.7cm)
node[left=13, below=26, label={0:$E_{g-1}$}]{};
\node at (3.25, 0.07) {$\bullet$};
\node at (3.3, 1.5) {$P_{g-1}$};
\draw [draw=black,ultra thin] (5.5,-0.3) -- (6.5,0.2);  
\draw [draw=black,ultra thin] (6,-0.3) -- (5,0.2);  
\node at (7.1,0) {\ldots}; 
\node at (5.2, 0.1) {$\bullet$};
\node at (5.17, 1.5) {$Q_{g-1}$};

\draw [draw=black,ultra thin] (8.5,-0.3) -- (7.5,0.2);  
\draw [draw=black,ultra thin] (8,-0.3) -- (9,0.2);  
\draw (9.7,0.2) ellipse (1cm and 0.7cm)
node[left=13, below=26, label={0:$E_{g}$}]{};
\node at (8.73, 0.07) {$\bullet$};
\node at (8.8, 1.5) {$P_{g}$};
\draw [draw=black,ultra thin] (11,-0.3) -- (12,0.2);  
\draw [draw=black,ultra thin] (11.5,-0.3) -- (10.5,0.2);  
\node at (12.6,0) {\;}; 
\node at (10.7, 0.1) {$\bullet$};
\node at (10.75, 1.5) {$Q_{g}$}; 
\end{tikzpicture}  

    \caption{\;}
    \label{fig:1}
\end{figure}

Of significance here is the following property exhibited by such families: When considering a dominant morphism $T'\rightarrow T$ of spectra of discrete valuation rings, the family $\mathscr{C}'\rightarrow T'$ acquired through base extension and minimal resolution of singularities results in a special fiber $\mathscr{C}_0'$ that satisfies the same conditions as $\mathscr{C}_0$.
 
 If necessary, we extend the base to assume the existence of a line bundle $\epsilon$ on $\mathscr{C}$ with $\epsilon^2\cong\mathscr{O}_{\mathscr{C}}$. In addition, its restriction $\epsilon_\eta$ on $\mathscr{C}_\eta$ is nontrivial, while the restriction to $\mathscr{C}_0$ is non-trivial only on $E_g$. We define $\widetilde{\mathscr{C}}:=\mathrm{Spec}(\mathscr{O}_{\mathscr{C}}\oplus\epsilon)$, and the ring structure on $\widetilde{\mathscr{C}}$ arises due to the isomorphism $\epsilon^2\cong\mathscr{O}_{\mathscr{C}}$. This ensures $\pi:\widetilde{\mathscr{C}}\rightarrow\mathscr{C}$ acts as an \'{e}tale double covering map over $T$, with the generic fiber $\widetilde{\mathscr{C}}_\eta$ smooth and geometrically irreducible, and the special fiber $\widetilde{\mathscr{C}}_0$, a reduced curve of arithmetic genus $2g-1$ depicted as Fig. \ref{fig:1}. That is, the special fiber consists of smooth components arranged in a chain, some being rational and others elliptic, with only ordinary double points as singularities. We denote these elliptic components by $E_i'$, $E_i''$ for $i=1,\ldots,g-1$, and $\widetilde{E}_g$. When the map $\pi$ is restricted over each curve $E_i$, it forms a reducible double covering $E_i'\sqcup E_i''\rightarrow E_i$, where both $E_i'$ and $E_i''$ are isomorphic to $E_i$. In addition, restricting $\pi$ over $E_g$ gives rise to an irreducible double covering $\widetilde{E}_g\rightarrow E_g$. Especially, the difference between the preimages of two points where $E_g$ intersects the remaining curve $\mathscr{C}_0$ corresponds to a $2$-torsion point in the Jacobian $\mathrm{Jac}(\widetilde{E}_g)$. 
 
 Then we carefully pick sections $P:T\rightarrow \widetilde{\mathscr{C}}$ and $Q:T\rightarrow \widetilde{\mathscr{C}}$ such that the corresponding points $P_0$ and $Q_0$ are in a rational component that splits the chain of $\widetilde{\mathscr{C}}_0$, resulting in a connected component with an arithmetic genus of $2g-1$.

Presented below is an illustrative diagram encapsulating the maps:
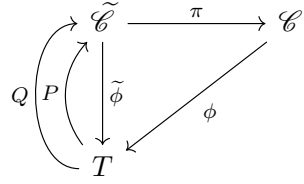
\[
\begin{tikzcd}
\widetilde{\mathscr{C}} \arrow[rr, "\pi"] \arrow[dd, "\widetilde{\phi}"]&& \mathscr{C} \arrow[ddll, "\phi"]\\
&&\\
T \arrow[uu, bend left=40,"P"] \arrow[uu, bend left=90,"Q"]&& 
\end{tikzcd}
\]

\subsection{On the quadruple}\label{sec:5.1}

To prove Theorem \ref{thm:main}, we need to establish Theorem \ref{thm:main2}, which necessitates certain preparatory setups and lemmas covered in the following sections: \S\ref{sec:5.1},\S\ref{sec:5.2}, and \S\ref{sec:5.3}.

The relevant coupled Prym-Petri map for Theorem \ref{thm:main2} can be described as follows. Let $V_{\mathbf{a},\mathbf{b}}^r(\mathscr{C}_\eta, P_\eta,Q_\eta)$ be the two-pointed Prym-Brill Noether loci associated to $\mathscr{C}_\eta$ with two points $P_\eta$ and $Q_\eta$. We consider $\mathscr{L}_\eta\in V_{\mathbf{a},\mathbf{b}}^r(\mathscr{C}_\eta,P_\eta,Q_\eta)$ and $\mathscr{M}_\eta:=\iota^*\mathscr{L}_\eta$. Here one may read $\mathscr{M}_\eta$ as $K_{\widetilde{\mathscr{C}}_\eta}\otimes\mathscr{L}_\eta^\vee$. 
Suppose 
$h^0(\widetilde{\mathscr{C}}_\eta,\mathscr{L}_\eta(-a_iP_\eta-b_iQ_\eta))=r+1-i$ for each $i$, and take sections 
\[
\sigma_{i}\in\widetilde{\phi}_*\mathscr{L}_\eta(-a_iP_\eta-b_iQ_\eta)\backslash  \widetilde{\phi}_*\mathscr{L}_\eta(-a_{i+1}P_\eta-b_{i+1}Q_\eta)\;\text{for}\;0\leq i\leq r-1
\]
and $\sigma_{r}\in \widetilde{\phi}_*\mathscr{L}_\eta(-a_{r}P_\eta-b_{r}Q_\eta)$.
As before, we have the composition
\begin{equation}\label{eqn:submo}
\iota^*\widetilde{\phi}_*\mathscr{L}_\eta(-a_iP_\eta-b_iQ_\eta)\hookrightarrow \iota^*\widetilde{\phi}_*\mathscr{L}_\eta\xrightarrow{\cong}\widetilde{\phi}_*\mathscr{M}_\eta \hookrightarrow \widetilde{\phi}_*\mathscr{M}_\eta(a_iP_\eta+b_iQ_\eta).
\end{equation}
The coupled Prym-Petri map for $\mathscr{L}_\eta$ becomes 
\begin{equation}\label{eqn:varmu}
\overline{\mu}_\eta:\displaystyle\bigoplus_{i=0}^r\langle \sigma_{i}\rangle \otimes \widetilde{\phi}_*\mathscr{M}_\eta(a_iP_\eta+b_iQ_\eta)/\iota^*\widetilde{\phi}_*\mathscr{L}_\eta(-a_iP_\eta-b_iQ_\eta)\rightarrow \widetilde{\phi}_*(K_{\mathscr{C}_\eta}\otimes \epsilon).
\end{equation}

The map $\overline{\mu}_\eta$ can be extended to a map over $T$. 
Since the map $\overline{\mu}_\eta$ is defined on a subspace of
\[
\displaystyle\bigoplus_{i=0}^r\langle \sigma_{i}\rangle \otimes \widetilde{\phi}_*\mathscr{M}_\eta(a_iP_\eta+b_iQ_\eta)\subseteq\displaystyle\bigoplus_{i=0}^r\widetilde{\phi}_*\mathscr{L}_\eta(-a_iP_\eta-b_iQ_\eta) \otimes \widetilde{\phi}_*\mathscr{M}_\eta(a_iP_\eta+b_iQ_\eta),
\]
we extend the line bundles $\mathscr{L}_\eta, \mathscr{L}_\eta(-a_iP_\eta-b_iQ_\eta)$ and $\mathscr{M}_\eta(a_iP_\eta+b_iQ_\eta)$ over $\widetilde{\mathscr{C}}$.

As in \cite[(2.7)]{Wel85}, following possible base extension and minimally resolving the singularities, one can extend the line bundle $\mathscr{L}_\eta$ on $\widetilde{\mathscr{C}}_\eta$ to a line bundle $\mathscr{L}$ on $\widetilde{\mathscr{C}}$. Here $\mathrm{Nm}(\mathscr{L})$ is in fact given by $K_{\mathscr{C}/T}$ twisted by a line bundle associated to a linear combination of the components of $\mathscr{C}_0$. We may assume that the component $E_g$ is not included in the linear combination, since the Picard group of $T$ is trivial so that $\mathscr{O}_{\mathscr{C}}(\mathscr{C}_0)\cong\mathscr{O}_{\mathscr{C}}$. Due to $\mathrm{Nm}(\mathscr{O}_{\widetilde{C}}(E_i'))=\mathrm{Nm}(\mathscr{\widetilde{C}}(E_i''))\cong\mathscr{C}(E_i)$ for $1\leq i\leq g-1$ and similarly for all the rational components of $\widetilde{\mathscr{C}}_0$, after a suitable twist, we can assume $\mathrm{Nm}(\mathscr{L})\cong K_{\mathscr{C}/T}$ from this point forward.

 Let $Y$ be any component of $\widetilde{\mathscr{C}}_0$. By using the theory of limit linear series \cite{EH83}, the line bundle $\mathscr{L}_\eta$ over $\widetilde{\mathscr{C}}_\eta$ can be extended to a line bundle $\mathscr{L}_Y$ over $\widetilde{\mathscr{C}}$.
 This extension involves twisting a line bundle $\mathscr{L}$ on $\widetilde{\mathscr{C}}$ with an appropriate linear combination of the components of $\widetilde{\mathscr{C}}_0$. The extended line bundle $\mathscr{L}_Y$ has degree $0$ on the components of the special fiber $\widetilde{\mathscr{C}}_0$ apart from $Y$. Similarly, we have line bundles over $\mathscr{C}$ that are extensions of line bundles on $\mathscr{C}_\eta$. Further we observe that  $\iota^*(\mathscr{L}_Y)\cong (\iota^*\mathscr{L})_{\iota(Y)}$ and $\mathrm{Nm}(\mathscr{L}_Y)\cong\mathrm{Nm}(\mathscr{L})_{\pi(Y)}$ as in \cite[pg. 677]{Wel85}.

 Additionally, as an analogous result from \cite[p.11]{Tar} for the pointed case, we have extensions for the two-pointed cases as follows. For each $i$, we consider the extensions $\mathscr{L}_{\widetilde{E}_g}^i$ of $\mathscr{L}_\eta(-a_iP_\eta-b_iQ_\eta)$, $\mathscr{M}_{\widetilde{E}_g}^i$ of $\mathscr{M}_\eta(a_iP_\eta+b_iQ_\eta)$, $\mathscr{L}_{\widetilde{E}_g}$ of $\mathscr{L}_\eta$, and $\mathscr{M}_{\widetilde{E}_g}$ of $\mathscr{M}_\eta$ over $\widetilde{\mathscr{C}}$ such that those bundles have degree $0$ on all the components of $\widetilde{\mathscr{C}}_0$ other than $\widetilde{E}_g$. Then we have inclusions
\begin{align*}
\widetilde{\phi}_*\mathscr{L}_{\widetilde{E}_g}^i\hookrightarrow \widetilde{\phi}_*\mathscr{L}_\eta(-a_iP_\eta-b_iQ_\eta)\quad\text{and}\quad \widetilde{\phi}_*\mathscr{M}_{\widetilde{E}_g}^i\hookrightarrow \widetilde{\phi}_*\mathscr{M}_\eta(a_iP_\eta+b_iQ_\eta)
\end{align*}
of $\mathscr{O}_T$-submodules.
Especially, given the inclusion
\[
\widetilde{\phi}_*\mathscr{L}_\eta(-a_{i+1}P_\eta-b_{i+1}Q_\eta)\subseteq\widetilde{\phi}_*\mathscr{L}_\eta(-a_iP_\eta-b_iQ_\eta),
\] 
there are integers $\alpha_{i}$ such that the conditions
\begin{align*}
t^{\alpha_{i}}\sigma_{i}&\in \widetilde{\phi}_*\mathscr{L}_{\widetilde{E}_g}^i\backslash\widetilde{\phi}_*\mathscr{L}_{\widetilde{E}_g}^{i+1}\quad\text{for $i=0,\ldots,r-1$}
\end{align*}
and $t^{\alpha_{r}}\sigma_{r}\in \widetilde{\phi}_*\mathscr{L}_{\widetilde{E}_g}^r$ are satisfied on $\widetilde{E}_g$ as in \cite[Lemma 1.2]{EH83}. We note that $\{t^{\alpha_i}\sigma_{i}\}_{i}$ forms a basis of $\widetilde{\phi}_*\mathscr{L}_{\widetilde{E}_g}^0$. 
We let 
\[
\widehat{\sigma}_{i}:=t^{\alpha_{i}}\sigma_{i}\quad\text{for all $i$}.
\]
Then, from the product of sections, one can deduce the map
\[
\overline{\mu}:\displaystyle\bigoplus_{i}\langle \widehat{\sigma}_{i}\rangle \otimes \widetilde{\phi}_*\mathscr{M}_{\widetilde{E}_g}^i/\iota^*\widetilde{\phi}_*\mathscr{L}_{\widetilde{E}_g}^i\rightarrow \widetilde{\phi}_*(K_{\mathscr{C}_\eta}\otimes \epsilon)
\]
and this provides \eqref{eqn:varmu} over $\eta$. In particular, we have the following inclusion
\begin{equation}\label{inc}
\iota^*\widetilde{\phi}_*\mathscr{L}_{\widetilde{E}_g}^i\hookrightarrow \widetilde{\phi}_*\mathscr{M}_{\widetilde{E}_g}^i
\end{equation}
of $\mathscr{O}_T$-submodules obtained by (\refeq{eqn:submo}).

\subsection{The Kernel of the coupled Prym-Petri map}\label{sec:5.2}
In this section, we would like to investigate a particular situation: Given $\rho\in \mathrm{ker}(\overline{\mu}_\eta)$, we assume $\rho\neq 0$. Then there is a unique integer $\gamma$ satisfying
\[
t^\gamma\rho\in\displaystyle\bigoplus_{i}\langle \widehat{\sigma}_{i}\rangle\otimes\widetilde{\phi}_*\mathscr{M}_{\widetilde{E}_g}^i\backslash t\left(\langle \widehat{\sigma}_{i}\rangle\otimes\widetilde{\phi}_*\mathscr{M}_{\widetilde{E}_g}^i\right).
\]
Indeed, $t^\gamma\rho$ lies in the kernel of $\overline{\mu}$, since the element $\rho$ is in the kernel of $\overline{\mu}_\eta$. For each $i$, we consider restrictions to
\begin{equation}\label{eqn:LM}
\begin{aligned}
L^{i}&:=\mathrm{Im}\left(\widetilde{\phi}_*\mathscr{L}_{\widetilde{E}_g}^i\rightarrow H^0\left(\mathscr{L}_{\widetilde{E}_g}^i\otimes\mathscr{O}_{\widetilde{E}_g}\right)\right),\\
M^{i}&:=\mathrm{Im}\left(\widetilde{\phi}_*\mathscr{M}_{\widetilde{E}_g}^i\rightarrow H^0\left(\mathscr{M}_{\widetilde{E}_g}^i\otimes\mathscr{O}_{\widetilde{E}_g}\right)\right).
\end{aligned}
\end{equation}

We denote by $\overline{\sigma}_{i}\in L^{i}$ the image of $\widehat{\sigma}_{i}$ under the restriction $\widetilde{\phi}_*\mathscr{L}_{\widetilde{E}_g}^i\rightarrow L^{i}$, and $\overline{\rho}$ the image of $t^\gamma\rho$ under the map
\[
\displaystyle\bigoplus_{i}\langle \widehat{\sigma}_{i}\rangle\otimes\widetilde{\phi}_*\mathscr{M}_{\widetilde{E}_g}^i\rightarrow\displaystyle\bigoplus_{i}\langle\overline{\sigma}_{i}\rangle\otimes M^{i}.
\] 
We know from the assumption that 
\begin{equation}\label{eqn:neq}
\overline{\rho}\neq 0\quad\text{in}\quad\displaystyle\bigoplus_{i}\langle\overline{\sigma}_{i}\rangle\otimes M^{i}/\iota^*L^{i},
\end{equation}
which is deduced by the inclusion $\iota^*L^{i}\hookrightarrow M^{i}$ via \eqref{inc}.

Before establishing the proof of Lemma \ref{lem:ine} which concerns the inequality involving the sum of the orders of $\overline{\rho}$ at $P_g'$ and $Q_g'$ (resp. $P_g''$ and $Q_g''$), we need Lemma \ref{lem:ord}. 
Let $X$ be a smooth surface and $\Phi:X\rightarrow T=\mathrm{Spec}(\mathcal{O})$ be a flat projective family. Here $\mathcal{O}$ is a discrete valuation ring whose parameter is $t$. Let $X_0$ is the fiber over a special point $t=0$.

\begin{lemma}\label{lem:ord}

Let $Y$ and $Z$ be two components of $X_0$ such that $Y$ and $Z$ meet at $p$. Let $q_1,q_2,\ldots,q_s$ be points on $Y$ for some $s$, where $q_i\neq p$. If $\alpha$ is the unique integer such that 
\[
t^\alpha\sigma\in\Phi_*\mathscr{L}_Z\backslash t\Phi_*\mathscr{L}_Z
\]
for a line bundle $\mathscr{L}_Z$ on $Z$, then 
\[
\mathrm{ord}_{q_1}(\sigma|_Y)+\ldots+\mathrm{ord}_{q_j}(\sigma|_Y)\leq \alpha
\]
for $1\leq j\leq s$.
\end{lemma}
\begin{proof}
It follows \cite[Proof of Proposition 1.1]{EH83} with 
\[
\mathrm{deg}\;\mathscr{L}_Y|_Y=\mathrm{ord}_{q_1}(\sigma|_Y)+\ldots+\mathrm{ord}_{q_j}(\sigma|_Y)+\mathrm{ord}_{p}(\sigma|_Y)+\sum_{q\in Y\backslash\{p,q_1,\ldots,q_j\}}\mathrm{ord}_q(\sigma|_Y).
\]
\end{proof}

Now, we are ready to show the following lemma. Let $P_g'$ and $Q_g'$ be intersection points of $\widetilde{E}_g$ with the adjacent connected components, and let $P_g''$ and $Q_g''$ be the images of $P_g'$ and $Q_g'$ under the involution $\iota$ respectively, as in Pig. \ref{fig:1}.

\begin{lemma}\label{lem:ine}
If $\rho\in \mathrm{ker}(\overline{\mu}_\eta)$, then we have
\begin{equation}\label{eqn:ord}
\mathrm{ord}_{P_g'}(\overline{\rho})+\mathrm{ord}_{Q_g'}(\overline{\rho})\geq 2g-2\quad\text{and}\quad\mathrm{ord}_{P_g''}(\overline{\rho})+\mathrm{ord}_{Q_g''}(\overline{\rho})\geq 2g-2.
\end{equation}
\end{lemma}

We say the relation \eqref{eqn:ord} on $\overline{\rho}$ if and only if $\overline{\rho}$ is a linear combination of elements of the form $\overline{\sigma}_i\otimes\overline{\tau}_j$ for $\overline{\tau}_j\in M^i$ where $\mathrm{ord}_{P_g'}(\overline{\sigma}_i)+\mathrm{ord}_{Q_g'}(\overline{\sigma}_i)+\mathrm{ord}_{P_g'}(\overline{\tau}_j)+\mathrm{ord}_{Q_g'}(\overline{\tau}_j)\geq 2g-2$ for all $i,j$.

We borrow some results in \cite[Proof of 3.2]{CHT},\cite[\S 3]{EH83}, and \cite[p. 679]{Wel85} along with \cite[Proof of Lemma 3.4]{Tar} to verify Lemma \ref{lem:ine}.

\begin{proof}[Proof of Lemma \ref{lem:ine}]
Let $1\leq k\leq g-1$. Let $P_k'$ and $Q_k'$ be the intersection points where $E_k'$ meets the adjacent components for $k=1,\ldots,g-1$, and the connected components meeting $E_k'$ at $P_k'$ contains the points $P_0'$ and $Q_0'$, as described in Fig. \ref{fig:1}. 

By the same argument for \cite[Lemma 1.2]{EH83}, it can be assumed that there is a suitable power $\alpha_{i,k}$ so that
\begin{equation}\label{eqn:al}
t^{\alpha_{i,k}}\sigma_i\in\widetilde{\phi}_*\mathscr{L}_{{E}_k'}(-a_iP_k'-b_iQ_k')\backslash\widetilde{\phi}_*\mathscr{L}_{{E}_k'}(-a_{i+1}P_k'-b_{i+1}Q_k')
\end{equation}
for $i=0,\ldots,r-1$, and $t^{\alpha_{r}}\sigma_{r}\in \widetilde{\phi}_*\mathscr{L}_{{E}_k'}(-a_rP_k'-b_rQ_k')$ on ${E}_k'$. For $0\leq i\leq r$, we let $\widehat{\sigma}_{i,k}:=t^{\alpha_{i,k}}\sigma_i$. Subsequently, there is a unique integer $\gamma_k$ for each $k$ such that 
\[
t^{\gamma_k}\rho\in\bigoplus_{i=0}^r\left(\langle\widehat{\sigma}_{i,k}\rangle\otimes\widetilde{\phi}_*\mathscr{M}_{{E}_k'}(a_iP_k'+b_iQ_k')\backslash t\left(\langle\widehat{\sigma}_{i,k}\rangle\otimes\widetilde{\phi}_*\mathscr{M}_{{E}_k'}(a_iP_k'+b_iQ_k')\right)\right).
\]
In particular, for $k=g$, we use $E_g', \widehat{\sigma}_{i,g}$ and $\gamma_g$ to indicate $\widetilde{E}_g,\widehat{\sigma}_i$, and $\gamma$ respectively. 

We modify some arguments for the unpointed case in \cite[Proof of Proposition 1.1, pp. 277-280]{EH83}, and the pointed case in \cite[Proof of 3.2]{CHT} as well as \cite[Proof of Lemma 3.3]{Tar}.

To be precise, by the construction of $\mathbf{a}$ and $\mathbf{b}$ (via (\ref{a},~\ref{b},~\ref{c})), \eqref{eqn:al} implies that for each $i$, there exists $\alpha_{i,k}$ that is $\alpha_{i,k}^{Q_k'}$ or $\alpha_{i,k}^{P_k'}$ such that either
\begin{align*}
t^{\alpha_{i,k}^{Q_k'}}\sigma_i&\in\widetilde{\phi}_*\mathscr{L}_{{E}_k'}(-a_u'P_k'-b_u'Q_k')\backslash\widetilde{\phi}_*\mathscr{L}_{{E}_k'}(-a_{u}'P_k'-b_{u+1}'Q_k')\quad\text{or}\\
t^{\alpha_{i,k}^{P_k'}}\sigma_i&\in\widetilde{\phi}_*\mathscr{L}_{{E}_k'}(-a_u'P_k'-b_u'Q_k')\backslash\widetilde{\phi}_*\mathscr{L}_{{E}_k'}(-a_{u+1}'P_k'-b_{u}'Q_k')
\end{align*}
for some $u:=u(i)\in\{0,\ldots,r\}$. Hence, $\left\{\alpha_{i,k}^{Q_k'}\right\}_{i\in I}\bigcup\left\{\alpha_{j,k}^{Q_k'}\right\}_{j\in J}$ with $I\sqcup J=\{0,\ldots,r\}$ is a basis of $\widetilde{\phi}_*\mathscr{L}_{{E}_k'}(-a_0P_k'-b_0Q_k')$.
Then there exists either $\gamma_{i,k}^{Q_k'}$ or $\gamma_{i,k}^{P_k'}$ such that we have 
\begin{align}
t^{\gamma_{i,k}^{Q_k'}}\rho_i&\in \langle\widehat{\sigma}_{i,k}\rangle\otimes\widetilde{\phi}_*\mathscr{M}_{{E}_k'}(a_u'P_k'+b_u'Q_k')\backslash t\left(\langle\widehat{\sigma}_{i,k}\rangle\otimes\widetilde{\phi}_*\mathscr{M}_{{E}_k'}(a_u'P_k'+b_u'Q_k')\right),\;\text{or}\label{eqn:aa}\\
t^{\gamma_{i,k}^{P_k'}}\rho_i&\in \langle\widehat{\sigma}_{i,k}\rangle\otimes\widetilde{\phi}_*\mathscr{M}_{{E}_k'}(a_u'P_k'+b_u'Q_k')\backslash t\left(\langle\widehat{\sigma}_{i,k}\rangle\otimes\widetilde{\phi}_*\mathscr{M}_{{E}_k'}(a_u'P_k'+b_u'Q_k')\right)\label{eqn:bb}.
\end{align}
Thus, we set $\gamma_{k}=\mathrm{max}_{i,j}\left\{\gamma_{i,k}^{Q_k'},\gamma_{j,k}^{P_k'}\right\}$. Also, for the first case \eqref{eqn:aa}, $t^\gamma\rho=0$ if $\gamma>\gamma_{i,k}^{Q_k'}$ and for \eqref{eqn:bb}, $t^\gamma\rho=0$ if $\gamma>\gamma_{i,k}^{P_k'}$. This enable us to have
\[
\gamma_k\leq \mathrm{ord}_{P_k'}(t^{\gamma_k}\rho|_{E_k'})+\mathrm{ord}_{Q_k'}(t^{\gamma_k}\rho|_{E_k'}).
\]

Furthermore, Lemma \ref{lem:ord} gives either
\begin{equation}\label{eqn:eqlor}
\begin{aligned}
\mathrm{ord}_{P_{k-1}'}(\sigma_{i}|_{E_{k-1}'})+\mathrm{ord}_{Q_{k-1}'}(\sigma_{i}|_{E_{k-1}'})&\leq \alpha_{i,k}^{P_k'}\quad\text{or}\\
\mathrm{ord}_{P_{k-1}'}(\sigma_{i}|_{E_{k-1}'})+\mathrm{ord}_{Q_{k-1}'}(\sigma_{i}|_{E_{k-1}'})&\leq \alpha_{i,k}^{Q_k'}
\end{aligned}
\end{equation}
for each $i$. Following the argument of \cite[Lemma 3.2]{EH83} with its proof as well as the proof of \cite[Proposition 3.1]{EH83}, we obtain
\[
\mathrm{ord}_{P_{k-1}'}(\rho_{i}|_{E_{k-1}'})+\mathrm{ord}_{Q_{k-1}'}(\rho_{i}|_{E_{k-1}'})\leq \gamma_k.
\]

In particular, if equality holds on \eqref{eqn:eqlor}, $\sigma_i$ vanishes on $E_{k-1}'$ only at $P_{k-1}'$ and $Q_{k-1}'$. Then by \cite[(2.8)]{Wel85}, there is at most one section of degree $\alpha_{i,k}$ vanishing only at $P_{k-1}'$ and $Q_{k-1}'$ up to scalars in $E_{k-1}'$. Hence, with similar arguments as in \cite[pp. 279-280]{EH83}, the assumption $\rho\in \mathrm{ker}(\overline{\mu}_\eta)$ gives rise to
\begin{equation}\label{eqn:inq}
\mathrm{ord}_{P_{k}'}(t^{\gamma_{k+1}}\rho|_{E_{k}'})+\mathrm{ord}_{Q_{k}'}(t^{\gamma_{k}}\rho|_{E_{k}'})\geq\mathrm{ord}_{P_{k-1}'}(t^{\gamma_{k-1}}\rho|_{E_{k-1}'})+\mathrm{ord}_{Q_{k-1}'}(t^{\gamma_{k-1}}\rho|_{E_{k-1}'})+2
\end{equation}
for $k=2,\ldots,g$. As mentioned in \cite[Proof of Lemma 3.4]{Tar} as well as \cite{Wel85}, the arguments presented in \cite{EH83,CHT} can be applied to the families of curves with special fibers consisting of a chain of rational and elliptic curves. So, our choice of points where the difference of any of two nodal points in $E_{k}'$ is non-torsion in $\mathrm{Jac}(E_k')$ enables us to have the inequality \eqref{eqn:inq}.

Thus, from \eqref{eqn:inq} for $k=2,\ldots,g$, we finally obtain the statement on $P_g'$ and $Q_g'$. Similarly, we have the other statement about $P_g''$ and $Q_g''$ as $P_l''$ and $Q_l''$ can be seen as $\iota(P_l')$ and $\iota(Q_l')$ with $E_l'':=\iota(E_l')$ for $l=1,\ldots,g$.
\end{proof}

We recall that the difference $P_g'-P_g''$ and $Q_g'-Q_g''$ are $2$-torsion points in $\mathrm{Jac}(\widetilde{E}_g)$. So, there is nothing we can achieve beyond \eqref{eqn:ord} from the above application of the argument presented in \cite{EH83,CHT,Tar}.

Lemma \ref{lem:ord} is analogous to \cite[Lemma 3.4]{Tar} for pointed case and \cite[(2.20)]{Wel85} for unpointed case. In the subsequent section, we provide a statement on the vanishing of $\overline{\rho}$, Lemma \ref{lem:M}.

\subsection{A vanishing statement}\label{sec:5.3}
For the points $P_g'$ and $Q_g'$ on $\widetilde{E}_g$, we have
\begin{equation}\label{eqn:lbi}
\begin{aligned}
\mathscr{L}_{\widetilde{E}_g}^i\otimes\mathscr{O}_{\widetilde{E}}&\cong\mathscr{L}_{\widetilde{E}_g}\otimes\mathscr{O}_{\widetilde{E}}(-a_iP'_\eta-b_iQ'_\eta),\;\text{and}\\
\mathscr{M}_{\widetilde{E}_g}^i\otimes\mathscr{O}_{\widetilde{E}}&\cong\mathscr{M}_{\widetilde{E}_g}\otimes\mathscr{O}_{\widetilde{E}}(a_iP'_\eta+b_iQ'_\eta)
\end{aligned}
\end{equation}
from $\mathscr{L}_\eta(-a_iP_\eta-b_iQ_\eta)$ and $\mathscr{M}_\eta(a_iP_\eta+b_iQ_\eta)$ discussed in \S\ref{sec:5.1}. It is shown in \cite[(2.21)]{Wel85} that both line bundles $\mathscr{L}_{\widetilde{E}_g}\otimes\mathscr{O}_{\widetilde{E}_g}$ and $\mathscr{M}_{\widetilde{E}_g}\otimes\mathscr{O}_{\widetilde{E}_g}$ exhibit isomorphism to either
\begin{equation}\label{eqn:ib}
\mathscr{O}_{\widetilde{E}_g}((2g-2)P_g')\quad\text{or}\quad\mathscr{O}_{\widetilde{E}_g}((2g-3)P_g'+P_g'').
\end{equation}
This is derived by the fact that these line bundles \eqref{eqn:ib} are different but have the same image as $\mathscr{O}_{E_g}((2g-2)P_g)$ under the $(2:1)$ norm map
\[
\mathrm{Nm}:\mathrm{Pic}^{2g-2}(\widetilde{E}_g)\rightarrow\mathrm{Pic}^{2g-2}(E_g),
\]
the line bundles \eqref{eqn:ib} are invariant under the involution $\iota$ of $\widetilde{E}_g$ from $2P_g'\equiv2P_g''$ in $\mathrm{Jac}(\widetilde{E}_g)$, and the isomorphisms
\[
\mathrm{Nm}(\mathscr{L}_{\widetilde{E}_g}\otimes\mathscr{O}_{\widetilde{E}_g})\cong\mathscr{O}_{E_g}((2g-2)P_g)\quad\text{and}\quad\mathrm{Nm}(\mathscr{M}_{\widetilde{E}_g}\otimes\mathscr{O}_{\widetilde{E}_g})\cong\mathscr{O}_{E_g}((2g-2)P_g).
\]
Putting together with \eqref{eqn:lbi}, we have the following lemma. 

\begin{lemma}\label{lem:M}
Let $\mathscr{L}_{\widetilde{E}_g}\otimes\mathscr{O}_{\widetilde{E}_g}$ and $\mathscr{M}_{\widetilde{E}_g}\otimes\mathscr{O}_{\widetilde{E}_g}$ be the line bundles both isomorphic either to one of two line bundles in \eqref{eqn:ib}. If
\[
\overline{\rho}\in\bigoplus_{i=0}^r\langle\overline{\sigma}_i\rangle\otimes M^i/\iota^*L^i
\] with the property \eqref{eqn:ord}, then $\overline{\rho}=0$.
\end{lemma}
\begin{proof}
Provided the assumption $\mathscr{L}_{\widetilde{E}_g}\otimes\mathscr{O}_{\widetilde{E}}\cong \mathscr{M}_{\widetilde{E}_g}\otimes\mathscr{O}_{\widetilde{E}_g}$ along with (\ref{eqn:LM},~\ref{eqn:lbi}), the spaces $L^i$ and $M^i$ for all $0\leq i\leq r$ are injected into $V:= H^0(\mathscr{M}_{\widetilde{E}_g}\otimes\mathscr{O}_{\widetilde{E}_g}(a_rP_g'+b_rQ_g'))$. To prove this lemma, we find an appropriate basis of $V$. 
\begin{enumerate}[align=right,itemindent=2em,labelsep=2pt,labelwidth=1em,leftmargin=0pt,nosep]
\item We first consider the case where $\mathscr{L}_{\widetilde{E}_g}\otimes\mathscr{O}_{\widetilde{E}_g}\cong\mathscr{M}_{\widetilde{E}_g}\otimes\mathscr{O}_{\widetilde{E}_g}\cong\mathscr{O}_{\widetilde{E}_g}((2g-2)P_g')$. We note that $2Q_g'\equiv2Q_g''$ on $\widetilde{E}_g$, and take points $Q_1,R_1\in \widetilde{E}_g\backslash\{P_g',P_g'',Q_g',Q_g''\}$ such that $Q_1+R_1\equiv Q_g'+Q_g''$. Let us define the divisors 
\begin{equation}\label{div}
\begin{aligned}
D_{2(k+l)}:=&(2g-2+a_r-2k)P_g'+2kP_g''+(b_r-2l)Q_g'+2lQ_g''\\
&\quad\text{for}\;k=0,\ldots,g-1+\left\lfloor \dfrac{a_r}{2}\right\rfloor, l=0,\ldots, \left\lfloor \dfrac{b_r}{2}\right\rfloor\\
D_{2(k+l)+1}:=&(2g-2+a_r-2k)P_g'+2kP_g''+(b_r-3-2l)Q_g'+(2l+1)Q_g''+Q_1+R_1\\
&\quad\text{for}\;k=0,\ldots,g-1+\left\lfloor \dfrac{a_r}{2}\right\rfloor,l=0,\ldots,\left\lfloor\dfrac{b_r-2}{2}\right\rfloor
\end{aligned}
\end{equation}
on $\widetilde{E}_g$. Let $e_n$ be a section in $V$ with divisor $D_n$ for $n\in \{0,\ldots,2g-4+a_r+b_r,2g-2+a_r+b_r\}$. We note that for any integers $a,b$, divisors $D_n'=D_n-aP_g'+aQ_g'-bP_g''+bQ_g''$ is isomorphic to $D_n$. For instance, it implies that $D_{2(1+0)}\cong D_{2(0+1)}$ because $D_{2(1+0)}=D_{2(0+1)}-2P_g'+2Q_g'$. Then we have $2g-2+a_r+b_r$ sections $e_n$ such that the sum of the possible vanishing orders at $P_g'$ and $Q_g'$ (resp. $P_g''$ and $Q_g''$) are all distinct. Thus $\{e_n\}_{n}$ forms a basis of $V$. We write
\[
\overline{\rho}=\bigoplus_{i=0}^r\overline{\rho}_i
\]
where $\overline{\rho}_i\in\langle \overline{\sigma}_i\rangle\otimes M^i$. We take $i$ satisfying $\overline{\rho}_i\neq 0$. By Lemma \ref{lem:ine}, we know that
\begin{align*}
\mathrm{ord}_{P_g'}(\overline{\rho}_i)+\mathrm{ord}_{Q_g'}(\overline{\rho}_i)&\geq 2g-2, \;\text{and}\\
\mathrm{ord}_{P_g''}(\overline{\rho}_i)+\mathrm{ord}_{Q_g''}(\overline{\rho}_i)&\geq 2g-2.
\end{align*}
The basis $\{e_n\}_n$ of $V$ leads to the following bases $\{\upsilon_h\}_h$ of $H^0(\mathscr{L}_{\widetilde{E}_g}\otimes\mathscr{O}_{\widetilde{E}}(-a_iP'_\eta-b_iQ'_\eta))$ and $\{\omega_m\}_m$ of $H^0(\mathscr{M}_{\widetilde{E}_g}\otimes\mathscr{O}_{\widetilde{E}}(a_iP'_\eta+b_iQ'_\eta))$.

We let $c_h=\mathrm{ord}_{P_g'}(\upsilon_h)+\mathrm{ord}_{Q_g'}(\upsilon_h)$. The way of constructing the basis using the divisors presented in \eqref{div} implicates the following equations
\begin{equation}\label{eqn:or1}
\mathrm{ord}_{P_g''}(\upsilon_h)+\mathrm{ord}_{Q_g''}(\upsilon_h)= \begin{cases}
2g-2-a_i-b_i-c_h &\text{if $a_i+b_i+c_h\equiv 0$ mod 2}\\
2g-4-a_i-b_i-c_h&\text{if $a_i+b_i+c_h \equiv 1$ mod 2}.
\end{cases}
\end{equation}
Also, if we assume $d_m=\mathrm{ord}_{P_g'}(\omega_m)+\mathrm{ord}_{Q_g'}(\omega_m)$, then we have
\begin{equation}\label{eqn:or2}
\mathrm{ord}_{P_g''}(\omega_m)+\mathrm{ord}_{Q_g''}(\omega_m)= \begin{cases}
2g-2+a_i+b_i-d_m &\text{if $a_i+b_i+d_m\equiv 0$ mod 2}\\
2g-4+a_i+b_i-d_m&\text{if $a_i+b_i+d_m \equiv 1$ mod 2}.
\end{cases}
\end{equation}

We know that $\langle \overline{\sigma}_i\rangle\subseteq H^0(\mathscr{L}_{\widetilde{E}_g}\otimes\mathscr{O}_{\widetilde{E}}(-a_iP'_\eta-b_iQ'_\eta))$ and $M^i\subseteq H^0(\mathscr{M}_{\widetilde{E}_g}\otimes\mathscr{O}_{\widetilde{E}}(a_iP'_\eta+b_iQ'_\eta))$. So, one can write $\overline{\rho}_i=\sum_{h,m}z_{h,m}\upsilon_h\otimes\omega_m$. As bases $\{\upsilon_h\}_h$ and $\{\omega_m\}_m$ have distinct sum of vanishing orders at $P_g'$ and $Q_g$ respectively, we have $z_{h,m}=0$ for all $h, m$ such that 
\begin{align*}
\mathrm{ord}_{P_g'}(\upsilon_h)+\mathrm{ord}_{P_g'}(\omega_m)+\mathrm{ord}_{Q_g'}(\upsilon_h)+\mathrm{ord}_{Q_g'}(\omega_m)&<2g-2,\;\text{and}\\
\mathrm{ord}_{P_g''}(\upsilon_h)+\mathrm{ord}_{P_g''}(\omega_m)+\mathrm{ord}_{Q_g''}(\upsilon_h)+\mathrm{ord}_{Q_g''}(\omega_m)&<2g-2.
\end{align*}
In fact, the conditions $\mathrm{ord}_{P_g'}(\upsilon_h\otimes\omega_m)+\mathrm{ord}_{Q_g'}(\upsilon_h\otimes\omega_m)\geq 2g-2$ and $\mathrm{ord}_{P_g''}(\upsilon_h\otimes\omega_m)+\mathrm{ord}_{Q_g''}(\upsilon_h\otimes\omega_m)\geq 2g-2$ happen when $c_h+d_m\geq 2g-2$ as well as
\begin{equation}\label{eqn:or3}
c_h+d_m=2g-2\quad\text{and}\quad a_i+b_i+c_h\equiv a_i+b_i+d_m\equiv 0\;\text{mod $2$}.
\end{equation}
Then we have
\begin{align*}
\mathrm{div}(\upsilon_h)&=u_1P_g'+(d_m-a_i-u_2)P_g'' +(c_h-u_1)Q_g'+(u_2-b_i)Q_g''\;\text{and}\\
\mathrm{div}(\omega_m)&=(d_m-u_2)P_g'+(a_i+u_1)P_g''+u_2Q_g'+(c_h+b_i-u_1)Q_g''
\end{align*}
for some nonnegative integers $u_1$ and $u_2$, where $0\leq u_1\leq c_h$ and $b_i\leq u_2\leq d_m-a_i$.
With these setups, we can conclude that the image of $\iota^*\upsilon_h$ under the composition of the following inclusions
\begin{align*}
\iota^*H^0\left(\mathscr{L}_{\widetilde{E}_g}\otimes\mathscr{O}_{\widetilde{E}_g}(-a_iP_g'-b_iQ_g')\right)\hookrightarrow\iota^*H^0\left(\mathscr{L}_{\widetilde{E}_g}\otimes\mathscr{O}_{\widetilde{E}_g}\right)\xrightarrow{\cong} H^0\left(\mathscr{M}_{\widetilde{E}_g}\otimes\mathscr{O}_{\widetilde{E}_g}\right)\\
\hookrightarrow H^0\left(\mathcal{M}_{\widetilde{E}_g}(a_iP_g'+b_iQ_g')\right)
\end{align*}
is contained in $\langle \omega_m\rangle$. This implies $\overline{\rho}_i$ in $\langle \overline{\sigma}_i\rangle\otimes\mathrm{Im}\left(\iota^*L^i\hookrightarrow M^i\right)$. Thus, it must be that $\overline{\rho}$ vanishes in $\bigoplus_{i=0}^r\langle\overline{\sigma}_i\rangle\otimes M^i/\iota^*L^i$, as desired.
\item Next, we take the case where $\mathscr{L}_{\widetilde{E}_g}\otimes\mathscr{O}_{\widetilde{E}_g}\cong\mathscr{M}_{\widetilde{E}_g}\otimes\mathscr{O}_{\widetilde{E}_g}\cong\mathscr{O}_{\widetilde{E}_g}((2g-3)P_g'+P_g'')$. Similarly, we have that $2Q_g'\equiv2Q_g''$ on $\widetilde{E}_g$ such that we have $Q_2,R_2\in \widetilde{E}_g\backslash\{P_g',P_g'',Q_g',Q_g''\}$ such that $Q_2+R_2\equiv P_g'+P_g''$.
We define divisors
\begin{equation}
\begin{aligned}
D_{2(k+l)+1}:=&(2g-3+a_r-2k)P_g'+(2k+1)P_g''+(b_r-2l)Q_g'+(2l+1)Q_g''\\
&\quad\text{for}\;k=0,\ldots,g+\left\lfloor \dfrac{a_r-3}{2}\right\rfloor,l=0,\ldots,\left\lfloor\dfrac{b_r}{2}\right\rfloor\\
D_{2(k+l)}:=&(2g-4+a_r-2k)P_g'+2kP_g''+(b_r-2l)Q_g'+2lQ_g''+Q_2+R_2\\
&\quad\text{for}\;k=0,\ldots,g-2+\left\lfloor \dfrac{a_r}{2}\right\rfloor, l=0,\ldots, \left\lfloor \dfrac{b_r}{2}\right\rfloor
\end{aligned}
\end{equation}
on $\widetilde{E}_g$. We follow the same argument as before, but using 
\begin{equation}
\mathrm{ord}_{P_g''}(\upsilon_h)+\mathrm{ord}_{Q_g''}(\upsilon_h)= \begin{cases}
2g-2-a_i-b_i-c_h &\text{if $a_i+b_i+c_h\equiv 1$ mod 2}\\
2g-4-a_i-b_i-c_h&\text{if $a_i+b_i+c_h \equiv 0$ mod 2}
\end{cases}
\end{equation}
rather than \eqref{eqn:or1}, and
\begin{equation}
\mathrm{ord}_{P_g''}(\omega_m)+\mathrm{ord}_{Q_g''}(\omega_m)= \begin{cases}
2g-2+a_i+b_i-d_m &\text{if $a_i+b_i+d_m\equiv 1$ mod 2}\\
2g-4+a_i+b_i-d_m&\text{if $a_i+b_i+d_m \equiv 0$ mod 2}
\end{cases}
\end{equation}
in place of \eqref{eqn:or2}. These modifications gives
\begin{equation*}
c_h+d_m=2g-2\quad\text{and}\quad a_i+b_i+c_h\equiv a_i+b_i+d_m\equiv 1\;\text{mod $2$}.
\end{equation*}
instead of \eqref{eqn:or3}. Then the rest of the arguments works as before. 
\end{enumerate}
\end{proof}

\subsection{Proof of Main result}

With the groundwork set from previous sections, we can now proceed to prove Theorem \ref{thm:main2}.

\begin{theorem}\label{thm:main2}
The quadruple $(\mathscr{C}_\eta,\epsilon_\eta,P_\eta,Q_\eta)$ satisfies the coupled Prym-Petri condition. 
\end{theorem}

\begin{proof}
 $\mathscr{L}_\eta\in V_{\mathbf{a},\mathbf{b}}^r(\mathscr{C}_\eta, P_\eta,Q_\eta)$ such that $h^0(\widetilde{C}_\eta,\mathscr{L}_\eta(-a_iP_\eta-b_iQ_\eta))=r+1-i$ for each $i$. Let $\rho$ be an element of the kernel of the coupled Prym-Petri map $\overline{\mu}_\eta$ \eqref{eqn:varmu}. Suppose $\rho\neq 0$. Then $\overline{\rho}\neq 0$ by \eqref{eqn:neq}. However, this gives a contradiction for Lemma \ref{lem:M} with Lemma \ref{lem:ine}. Hence, we obtain $\rho=0$.
\end{proof}

The following is our main result, the {\it coupled Prym-Petri Theorem}.

\begin{theorem}\label{thm:main}
Let $C$ be a general curve of genus $g$ and $\epsilon$ an arbitrary non-trivial $2$-torsion point in the Jacobian $\mathrm{Jac}(C)$. For general points $P$ and $Q$ in the \'{e}tale double cover $\widetilde{C}$, the quadruple $(C,\epsilon,P,Q)$ satisfies the coupled Prym-Petri condition.
\end{theorem}
\begin{proof}
Since it is an open condition on families of quadruple $(C,\epsilon, P,Q)$ to satisfy the coupled Prym-Petri condition, the statement can be proved by the existence of a quadruple $(C,\epsilon, P, Q)$ with $\epsilon\neq0$ where the coupled Prym-Petri condition is satisfied. This is based on the fact that the moduli space of quadruple $(C,\epsilon, P,Q)$ with a fixed genus and $\epsilon\neq 0$ is irreducible, as in \cite[Lemma 3.2]{Tar}.

So the statement can be established by proving that the geometric generic fiber satisfies the coupled Prym-Petri condition. Let us consider the geometric generic fiber $(\mathscr{C}_{\overline{\eta}},\epsilon_{\overline{\eta}}, P_{\overline{\eta}},Q_{\overline{\eta}})$, where $\mathscr{C}_{\overline{\eta}}:=\mathscr{C}_\eta\otimes\overline{k(\eta)}$, $\epsilon_{\overline{\eta}}:=\epsilon_\eta\otimes\overline{k(\eta)}$, and $P_{\overline{\eta}}$, $Q_{\overline{\eta}}$ are the points in $\widetilde{\mathscr{C}}_{\overline{\eta}}:=\widetilde{\mathscr{C}}_{\overline{\eta}}\otimes\overline{k(\eta)}$ induced by $P$ and $Q$ respectively. In fact, it is evident from \cite[p. 272]{EH83} that any line bundle on $\mathscr{C}_{\overline{\eta}}$ is constructed from a line bundle over some finite extension of $k(\eta)$. Henceforth, Theorem \ref{thm:main2} concludes the statement by finite base change and change of notation.
\end{proof}

Combined with Proposition \ref{prop:4.4} and Theorem \ref{thm:main}, we complete our second main Theorem \ref{main:2}.

As a Corollary \ref{cor}, we have the condition on the non-emptiness of the pointed Prym-Brill-Noether loci $V_{\mathbf{a},\mathbf{b}}^r(C,P,Q)$.

In Corollary \ref{cor}, if $g-1=|\mathbf{a}+\mathbf{b}|$, then $V_{\mathbf{a},\mathbf{b}}^r(C,\epsilon,P,Q)$ associated to a general curve $C$ with general points $P, Q$ is given by a finite number of distinct points. This number can be expressed by the degree of the class of $V_{\mathbf{a},\mathbf{b}}^r(C,\epsilon,P,Q)$. By the Poincar\'{e} formula, $\mathrm{deg}(\xi^{g-1})=(g-1)!=|\mathbf{a}+\mathbf{b}|!$ so that the degree of the class turns out to be 
\begin{equation}\label{eqn:class2}
\mathrm{deg}\left(\left[V_{\mathbf{a},\mathbf{b}}^r(C,\epsilon,P',Q')\right]\right)={|a+b|!}{2^{|\mathbf{a}+\mathbf{b}|-\ell_\circ}}\prod_{i=0}^{r}\dfrac{1}{{(a_i+b_i)!}}\prod_{j<i}\dfrac{a_i+b_i-a_j-b_j}{a_i+b_i+a_j+b_j}.
\end{equation}

Let $\varsigma=(\varsigma_1>\ldots>\varsigma_\ell)$ be a strict partition such that $\varsigma_\ell>0$. The strict partition is corresponding to a shifted shape consisting of $\varsigma_i$ boxes in $i$-th row where each $i$-th row is shifted $i$ steps to the right. The {\it standard Young tableau (SYT) of shifted shape} $\varsigma$ is a filling of $\varsigma$ using $1,\ldots,|\varsigma|$ such that the entries in each row and column are strictly increasing. For instance, if $\varsigma=(4,2,1)$, all the possibilities of SYT of $\varsigma$ are
\[{\small
{ \Tableau{
1&2&4&7\\
~&3&5&~\\
~&~&6\\
}}\quad{ \Tableau{
1&2&4&6\\
~&3&5&~\\
~&~&7\\
}}\quad{ \Tableau{
1&2&4&5\\
~&3&6&~\\
~&~&7\\
}}\quad{ \Tableau{
1&2&3&7\\
~&4&5&~\\
~&~&6\\
}}
\quad{ \Tableau{
1&2&3&6\\
~&4&5&~\\
~&~&7\\
}}
\quad{ \Tableau{
1&2&3&5\\
~&4&6&~\\
~&~&7\\
}}
\quad{ \Tableau{
1&2&3&4\\
~&5&6&~\\
~&~&7\\
}}}.
\]

As in \cite[pg. 3]{Tar}, the degree \eqref{eqn:class2} can be obtained by 
\[
2^{|\mathbf{a}+\mathbf{b}|-\ell_\circ}\#\{\text{SYT of Shifted shape}\;(a_r+b_r,\ldots,a_0+b_0)\}.
\]
In addition, if $a_0+b_0=0$, we use the strict partition $(a_r+b_r>\ldots>a_1+b_1)$ with $\ell_0=r$ for the degree of the class.

\bibliographystyle{amsplain}

\begin{bibdiv}
\begin{biblist}

\bib{And19}{article}{
   author={Anderson, David},
   title={$K$-theoretic Chern class formulas for vexillary degeneracy loci},
   journal={Adv. Math.},
   volume={350},
   date={2019},
   pages={440--485},
   issn={0001-8708},
}


\bib{ACT}{article}{
   author={Anderson, Dave},
   author={Chen, Linda},
   author={Tarasca, Nicola},
   title={$K$-classes of Brill-Noether loci and a determinantal formula},
   journal={Int. Math. Res. Not. IMRN},
   date={2022},
   number={16},
   pages={12653--12698},
}



\bib{AF20}{article}{
   author={Anderson, David},
   author={Fulton, William},
   title={Vexillary signed permutations revisited},
   journal={Algebr. Comb.},
   volume={3},
   date={2020},
   number={5},
   pages={1041--1057},
}


\bib{AIJK}{article}{
   author={Anderson, David},
   author={Ikeda, Takeshi},
   author={Jeon, Minyoung},
   author={Kawago, Ryotaro},
   title={The multiplicity of a singularity in a vexillary Schubert variety},
   journal={Adv. Math.},
   volume={435},
   date={2023},
   pages={Paper No. 109366, 39},
}
  
  \bib{ACGH}{book}{
   author={Arbarello, E.},
   author={Cornalba, M.},
   author={Griffiths, P. A.},
   author={Harris, J.},
   title={Geometry of algebraic curves. Vol. I},
   series={Grundlehren der mathematischen Wissenschaften [Fundamental
   Principles of Mathematical Sciences]},
   volume={267},
   publisher={Springer-Verlag, New York},
   date={1985},
   pages={xvi+386},
}

\bib{Cai}{article}{
   author={Cai, Shuang},
   title={Algebraic connective $K$-theory and the niveau filtration},
   journal={J. Pure Appl. Algebra},
   volume={212},
   date={2008},
   number={7},
   pages={1695--1715},
}


\bib{CHT}{article}{
   author={Ciliberto, Ciro},
   author={Harris, Joe},
   author={Teixidor i Bigas, Montserrat},
   title={On the endomorphisms of ${\rm Jac}(W^1_d(C))$ when $\rho=1$ and
   $C$ has general moduli},
   conference={
      title={Classification of irregular varieties},
      address={Trento},
      date={1990},
   },
   book={
      series={Lecture Notes in Math.},
      volume={1515},
      publisher={Springer, Berlin},
   },
   date={1992},
   pages={41--67},
}

\bib{DL}{article}{
   author={Dai, Shouxin},
   author={Levine, Marc},
   title={Connective algebraic $K$-theory},
   journal={J. K-Theory},
   volume={13},
   date={2014},
   number={1},
   pages={9--56},
}

\bib{DCP}{article}{
   author={De Concini, Corrado},
   author={Pragacz, Piotr},
   title={On the class of Brill-Noether loci for Prym varieties},
   journal={Math. Ann.},
   volume={302},
   date={1995},
   number={4},
   pages={687--697},
}
\bib{EH83}{article}{
   author={Eisenbud, D.},
   author={Harris, J.},
   title={A simpler proof of the Gieseker-Petri theorem on special divisors},
   journal={Invent. Math.},
   volume={74},
   date={1983},
   number={2},
   pages={269--280},
}
\bib{Ful92}{article}{
   author={Fulton, William},
   title={Flags, Schubert polynomials, degeneracy loci, and determinantal
   formulas},
   journal={Duke Math. J.},
   volume={65},
   date={1992},
   number={3},
   pages={381--420},
}



\bib{HIMN}{article}{
   author={Hudson, Thomas},
   author={Ikeda, Takeshi},
   author={Matsumura, Tomoo},
   author={Naruse, Hiroshi},
   title={Degeneracy loci classes in $K$-theory---determinantal and Pfaffian
   formula},
   journal={Adv. Math.},
   volume={320},
   date={2017},
   pages={115--156},
}


\bib{HIMN20}{article}{
   author={Hudson, Thomas},
   author={Ikeda, Takeshi},
   author={Matsumura, Tomoo},
   author={Naruse, Hiroshi},
   title={Double Grothendieck polynomials for symplectic and odd orthogonal
   Grassmannians},
   journal={J. Algebra},
   volume={546},
   date={2020},
   pages={294--314},
}



\bib{Jeon}{article}{
   author={Jeon, Minyoung},
   title={Euler characteristics of Brill-Noether loci on Prym varieties},
   journal={Manuscripta Math.},
   volume={173},
   date={2024},
   number={1-2},
   pages={753--769},
}


\bib{Mum}{article}{
   author={Mumford, David},
   title={Prym varieties. I},
   conference={
      title={Contributions to analysis (a collection of papers dedicated to
      Lipman Bers)},
   },
   book={
      publisher={Academic Press, New York},
   },
   date={1974},
   pages={325--350},
}


\bib{Mum2}{article}{
   author={Mumford, David},
   title={Theta characteristics of an algebraic curve},
   journal={Ann. Sci. \'{E}cole Norm. Sup. (4)},
   volume={4},
   date={1971},
   pages={181--192},
}


	
 \bib{Pflu}{article}{
author={Nathan Pflueger},
      title={Versality of Brill-Noether flags and degeneracy loci of twice-marked curves}, 
      journal={To appear in Algebraic Geometry, arXiv:2103.10969},
        date={2021},
 }
 
 
 \bib{Ra}{article}{
   author={Ramanathan, A.},
   title={Schubert varieties are arithmetically Cohen-Macaulay},
   journal={Invent. Math.},
   volume={80},
   date={1985},
   number={2},
   pages={283--294},
}
 
 
\bib{Sho}{article}{
   author={Shokurov, V. V.},
   title={Prym varieties: theory and applications},
   language={Russian},
   journal={Izv. Akad. Nauk SSSR Ser. Mat.},
   volume={47},
   date={1983},
   number={4},
   pages={785--855},
}
 
 
 
 \bib{Tar}{article}{
   author={Tarasca, Nicola},
   title={A pointed Prym-Petri theorem},
   journal={Trans. Amer. Math. Soc.},
   volume={376},
   date={2023},
   number={4},
   pages={2641--2656},
}





\bib{Wel85}{article}{
   author={Welters, Gerald E.},
   title={A theorem of Gieseker-Petri type for Prym varieties},
   journal={Ann. Sci. \'{E}cole Norm. Sup. (4)},
   volume={18},
   date={1985},
   number={4},
   pages={671--683},
}



\end{biblist}
\end{bibdiv}


\end{document}